%% file: relaxation_RK_entropy__arXiv.tex
\documentclass[a4paper]{scrartcl}

\usepackage[utf8]{luainputenc}
\usepackage[USenglish]{babel}
\usepackage{csquotes}

\usepackage[a4paper, top=2.5cm, bottom=2.5cm, left=2.5cm, right=2.5cm]{geometry}
\usepackage[plainpages=false,pdfpagelabels,hidelinks,unicode]{hyperref}

\usepackage[%
  backend=biber,
  style=numeric-comp,
  giveninits=true, uniquename=init, 
  natbib=true,
  url=true,
  doi=true,
  isbn=false,
  backref=false,
  maxnames=99,
  ]{biblatex}
\usepackage{amsmath}
\numberwithin{equation}{section}
\allowdisplaybreaks
\usepackage{amssymb}
\usepackage{commath}
\usepackage{mathtools}
\usepackage{bbm}
\usepackage{nicefrac}

\usepackage{amsthm}
\theoremstyle{plain}
  \newtheorem{theorem}{Theorem}
  \newtheorem{lemma}[theorem]{Lemma}
  \newtheorem{corollary}[theorem]{Corollary}
  \newtheorem{proposition}[theorem]{Proposition}
\theoremstyle{definition}
  
  \newtheorem{remark}[theorem]{Remark}
  
\numberwithin{theorem}{section}

\usepackage{color}
\usepackage{graphicx}
\usepackage[small]{caption}
\usepackage{subcaption}

\begingroup\expandafter\expandafter\expandafter\endgroup
\expandafter\ifx\csname pdfsuppresswarningpagegroup\endcsname\relax
\else
\fi

\usepackage{booktabs}
\usepackage{rotating}
\usepackage{multirow}

\usepackage{ifluatex}
\ifluatex
  \usepackage[no-math]{fontspec}
\else
  \usepackage[T1]{fontenc}
\fi
\usepackage{newpxtext,newpxmath}

\newcommand{\R}{\mathbb{R}}
\renewcommand{\H}{\mathcal{H}}
\renewcommand{\O}{\mathcal{O}}
\newcommand{\dt}{{\Delta t}}
\newcommand{\1}{\mathbbm{1}}
\newcommand{\e}{\mathrm{e}}
\newcommand{\scp}[2]{\left\langle{#1,\, #2}\right\rangle}

\newenvironment{keywords}{\par\textbf{Key words.}}{\par}
\newenvironment{AMS}{\par\textbf{AMS subject classification.}}{\par}

\include{definitions}

\title{Relaxation Runge--Kutta Methods: Fully-Discrete Explicit Entropy-Stable
       Schemes for the Compressible Euler and Navier--Stokes Equations}
\author{Hendrik Ranocha \and
        Mohammed Sayyari \and
        Lisandro Dalcin \and
        Matteo Parsani \and
        David I. Ketcheson}
\date{October 23, 2019}

\makeatletter
\hypersetup{pdftitle={\@title}}
\makeatother

\begin{document}

\maketitle

\input{relaxation_RK_entropy}

\printbibliography

\end{document}

%% file: definitions.tex
\usepackage{gensymb}
\usepackage{tikz}
\usetikzlibrary{matrix,backgrounds,shapes}
\usetikzlibrary{positioning}
\usetikzlibrary{fit}
\usetikzlibrary{plotmarks}
\usetikzlibrary{decorations.pathreplacing}
\usepackage{pgfplots}
\usetikzlibrary{shapes}
\usetikzlibrary{positioning,arrows}
\usetikzlibrary{fadings,shapes.arrows,shadows}
\usetikzlibrary{arrows.meta}
\usepackage{bm}
\usepackage{xspace}
\usepackage{afterpage}
\usepackage{color}
\usepackage{float}
\definecolor{darkorange}{rgb}{1.0, 0.55, 0.0}
\definecolor{royalblue}{rgb}{0.254,0.41,0.88}
\usepackage{xcolor}

\newcommand{\Tr}{\ensuremath{^{\mr{\top}}}}

\newcommand{\mr}[1]{\ensuremath{\mathrm{#1}}}
\newcommand{\fnc}[1]{\ensuremath{\mathcal{#1}}}

\newcommand{\mat}[1]{\ensuremath{\mathsf{#1}}}

\newcommand{\xm}[0]{\ensuremath{x_{m}}}
\newcommand{\xone}[0]{x_{1}}
\newcommand{\xtwo}[0]{x_{2}}
\newcommand{\xthree}[0]{x_{3}}

\newcommand{\M}[0]{\ensuremath{\overline{\mat{P}}}}
\newcommand{\Mk}[0]{\ensuremath{\mat{P}}_{j}}

\newcommand{\Q}[0]{\ensuremath{\bm{\fnc{Q}}}}
\newcommand{\W}[0]{\ensuremath{\bm{\fnc{W}}}}
\newcommand{\Fxm}[0]{\ensuremath{\bm{\fnc{F}}_{\xm}^{(I)}}}
\newcommand{\Fxmv}[0]{\ensuremath{\bm{\fnc{F}}_{\xm}^{(V)}}}
\newcommand{\GB}[0]{\ensuremath{\bm{\fnc{G}}^{(B)}}}
\newcommand{\Gzero}[0]{\ensuremath{\bm{\fnc{G}}^{(0)}}}
\newcommand{\Uone}[0]{\ensuremath{\fnc{U}_{1}}}
\newcommand{\Utwo}[0]{\ensuremath{\fnc{U}_{2}}}
\newcommand{\Uthree}[0]{\ensuremath{\fnc{U}_{3}}}

\newcommand{\Cij}[2]{\ensuremath{\mat{C}_{#1,#2}}}

\newcommand{\qk}[0]{\ensuremath{\bm{q}_{j}}}

\newcommand{\matJk}[0]{\ensuremath{\mat{J}}_{j}}

\newcommand{\onek}[0]{\ensuremath{\bm{1}}}

\newcommand{\nxm}[0]{\ensuremath{n_{\xm}}}

\newcommand{\Pmatvol}[0]{\ensuremath{\mat{P}}}


%% file: relaxation_RK_entropy.tex
\begin{abstract}
  The framework of inner product norm preserving relaxation Runge--Kutta methods
  (David I. Ketcheson, \emph{Relaxation {R}unge--{K}utta Methods: Conservation and
  Stability for Inner-Product Norms}, SIAM Journal on Numerical Analysis, 2019)
  is extended to general convex quantities.  Conservation, dissipation, or other
  solution properties with respect to any convex functional are enforced by
  the addition of a {\em relaxation parameter} that multiplies the Runge--Kutta
  update at each step.
  Moreover, other desirable stability (such as strong stability preservation)
  and efficiency (such as low storage requirements) properties are preserved.
  The technique can be applied to both explicit and implicit Runge--Kutta methods
  and requires only a small modification to existing implementations.
  The computational cost at each step is the solution of one additional scalar
  algebraic equation for which a good initial guess is available.
  The effectiveness of this approach is proved analytically and demonstrated in
  several numerical examples, including applications to high-order
  entropy-conservative and entropy-stable semi-discretizations on unstructured
  grids for the compressible Euler and Navier--Stokes equations.
\end{abstract}



\sloppy

\section{Introduction}

Consider a time-dependent ordinary differential equation (ODE)
\begin{equation}
\label{eq:ode}
\begin{aligned}
  \od{}{t} u(t) &= f(t, u(t)),
  && t \in (0,T),
  \\
  u(0) &= u^0,
\end{aligned}
\end{equation}
in a real Hilbert space $\H$ with inner product $\scp{\cdot}{\cdot}$, inducing
the norm $\norm{\cdot}$.  Let $\eta\colon \H \to \R$ denote a smooth convex function
whose correct evolution in time is important in the solution of \eqref{eq:ode}.
In relevant applications $\eta$ might represent e.g. some form of energy or momentum;
in the present work we refer to $\eta$ as \emph{entropy}, with a view to applications
in hyperbolic and incompletely parabolic system of partial differential equations (PDEs)
such as the compressible Euler and Navier--Stokes equations.
The time evolution of $\eta$ is given by
$\od{}{t} \eta(u(t)) = \scp{\eta'(u(t))}{f(t, u(t))}$.
Thus entropy \emph{dissipative} systems satisfy
\begin{equation}
\label{eq:dissipative}
  \forall u \in \H, t \in [0,T]\colon
  \quad
  \scp{\eta'(u)}{f(t, u)} \leq 0,
\end{equation}
while entropy \emph{conservative} ones fulfill
\begin{equation}
\label{eq:conservative}
  \forall u \in \H, t \in [0,T]\colon
  \quad
  \scp{\eta'(u)}{f(t, u)} = 0.
\end{equation}
In many applications it is important to preserve this qualitative behavior;
i.e. to ensure that
$$
    \eta(u^{n+1}) \le \eta(u^n)
$$
for a dissipative problem or that
$$
    \eta(u^{n+1}) = \eta(u^0)
$$
for a conservative problem.  Violation of these properties can lead to solutions
that are unphysical and qualitatively incorrect.  Nevertheless, most numerical
methods fail to guarantee these discrete properties.  In the present work, we
present a modification that makes any Runge--Kutta (RK) method preserve conservation
or dissipativity while also retaining other important and desirable properties
of the unmodified Runge--Kutta method.

In this work we focus on applications to entropy conservative or entropy dissipative
semi-discretizations of hyperbolic conservation laws \cite{tadmor1987numerical,tadmor2003entropy}
and the Navier--Stokes equations (incompletely parabolic).
Nevertheless, the methods presented here may be useful in many other applications,
including Hamiltonian systems, dispersive wave equations, and other areas where
geometric numerical integration is important.

\begin{remark}
  It is possible to generalize this setting to Banach spaces instead of Hilbert
  spaces. In that case, scalar products of the form $\scp{\eta'}{f}$ should be
  read as the application of the bounded linear functional $\eta'$ to $f$.
\end{remark}

\subsection{Related Work}

Recently, there has been some interest in nonlinear and entropy stability of numerical methods
for balance laws.
Several major hurdles remain on the path towards complete nonlinear and entropy stability of numerical algorithms
because
most of the research has been focused on semi-discrete
schemes (see, for instance, \cite{svard_entropy_stable_euler_wall_2014,parsani_entropy_stability_solid_wall_2015,carpenter_entropy_stability_ssdc_2016,Winters2016,ranocha2017shallow,Wintermeyer2017,fernandez_pref_euler_entropy_stability_2019,fernandez_pref_ns_entropy_stability_2019}).
Stability/dissipation results for fully discrete schemes have mainly been limited
to semi-discretizations including certain amounts of dissipation
\cite{tadmor2003entropy, higueras2005monotonicity, zakerzadeh2015high,
ranocha2018stability}, linear equations \cite{tadmor2002semidiscrete,
ranocha2018L2stability, sun2017stability, sun2018strong}, or fully implicit
time integration schemes \cite{tadmor2003entropy, friedrich2018entropy, lefloch2002fully}. For
explicit methods and general equations, there are negative experimental and
theoretical results concerning entropy stability \cite{ranocha2018strong,
lozano2018entropy}.

While applications to entropy conservative/dissipative schemes for hyperbolic
and parabolic
balance laws are included in this article, the general technique is not
limited to this setting but can be applied to many ordinary differential
equations, and to both explicit and implicit Runge--Kutta methods.
Since the basic idea is to preserve properties given at the continuous level
discretely, these schemes are related to the topic of geometric numerical
integration, see \cite{hairer2006geometric} and references therein.

The basic idea behind the methods proposed here comes from Dekker \& Verwer
\cite[pp. 265-266]{dekker1984stability} and has been developed for inner-product
norms in \cite{ketcheson2019relaxation}.
The idea (and notation) of Dekker \& Verwer \cite{dekker1984stability} was applied
in \cite{delbuono2002explicit} to a restricted class of fourth order methods.
This was extended in \cite{calvo2006preservation} by giving a general
proof that applying the technique to a Runge--Kutta method of order $p$ results in a method
of order at least $p-1$.  The idea was referred to therein as the \emph{incremental
direction technique} (IDT), and viewed as a Runge--Kutta projection method where
the search direction is the same as the direction of the next time step update.
Nevertheless, the main focus in \cite{calvo2006preservation}
is on a different type of projection in which the search direction is chosen based
on an embedded method.
Grimm \& Quispel \cite{grimm2005geometric} extended the standard orthogonal projection
method \cite[Section~IV.4]{hairer2006geometric} to dissipative systems possessing
a Lyapunov function and the same approach was used in \cite{calvo2010projection,
laburta2015numerical} with the choice of search directions advocated in
\cite{calvo2006preservation}. Kojima \cite{kojima2016invariants} reviewed some
related methods and proposed another kind of projection scheme for conservative
systems.

Like standard Runge--Kutta methods, and in contrast to orthogonal
projection methods, the schemes
based on the approach of Dekker \& Verwer \cite{dekker1984stability} or
Calvo et al. \cite{calvo2006preservation} preserve linear invariants --
a feature that is absolutely essential in the numerical solution of hyperbolic
conservation laws.
It is also interesting to study different projection methods since the behavior of
these schemes can depend crucially on the choice of conserved quantities
\cite[Section~IV.4]{hairer2006geometric} and the type of projection or search
direction \cite{calvo2010projection, calvo2015runge, kojima2016invariants,
ranocha2016time}.

The goal of this article is to extend the theory developed in
\cite{ketcheson2019relaxation} to a much broader class of problems. The resulting
schemes are shown to possess desirable properties, both theoretically and
in numerical experiments. In particular, applications include fully-discrete
entropy stable numerical methods of any order for the three-dimensional compressible Euler and
Navier--Stokes equations on unstructured grids based on summation-by-parts (SBP)
operators \cite{carpenter_ssdc_2014,Parsani2016}.
Analytical and numerical comparisons with other types of
projection schemes are left for future work.

\subsection{Runge--Kutta Methods}

A general (explicit or implicit) Runge--Kutta method with $s$ stages can be represented
by its Butcher tableau \cite{butcher2008numerical, hairer2008solving}
\begin{equation}
\label{eq:butcher}
\begin{array}{c | c}
  c & A
  \\ \hline
    & b^T
\end{array}\, ,
\end{equation}
where $A \in \R^{s \times s}$ and $b, c \in \R^s$. For \eqref{eq:ode}, a step
from $u^n \approx u(t_n)$ to $u^{n+1} \approx u(t_{n+1})$ where $t_{n+1} = t_n + \dt$
is given by
\begin{subequations}
\label{eq:RK-step}
\begin{align}
\label{eq:RK-stages}
  y_i
  &=
  u^n + \dt \sum_{j=1}^{s} a_{ij} \, f(t_n + c_j \dt, y_j),
  \qquad i \in \set{1, \dots, s},
  \\
\label{eq:RK-final}
  u^{n+1}
  &=
  u^n + \dt \sum_{i=1}^{s} b_{i} \, f(t_n + c_i \dt, y_i).
\end{align}
\end{subequations}
Here, $y_i$ are the stage values of the Runge--Kutta method.
We will make use of the shorthand
\begin{align}
  f_i & := f(t_n + c_i \dt, y_i), &
  f_0 & := f(t_n, u^n).
\end{align}
As is common in the literature, we assume that $A \1 = c$ with $\1 = (1, \dots, 1)^T \in \R^s$.

A Runge--Kutta method is (entropy) \emph{dissipation preserving} if
$\eta(u^{n+1}) \leq \eta(u^n)$
whenever the right hand side fulfills \eqref{eq:dissipative}. Similarly, it is
(entropy) \emph{conservative} if
$\eta(u^{n+1}) = \eta(u^n)$
whenever the system satisfies \eqref{eq:conservative}. Depending on the context,
such schemes are also called \emph{monotone} or \emph{strongly stable}
\cite{higueras2005monotonicity, ranocha2018strong}.

\section{Relaxation Runge--Kutta Methods}
\label{sec:RRK}

Following \cite[pp.~265--266]{dekker1984stability} and \cite{ketcheson2019relaxation},
the basic idea to make a given Runge--Kutta method entropy stable is to scale the
weights $b_i$ by a parameter $\gamma_n \in \R$, i.e. to use
\begin{equation}
\label{eq:RK-final-gamma}
  u^{n+1}_\gamma
  :=
  u^n + \gamma_n \dt \sum_{i=1}^{s} b_{i} f_i
\end{equation}
instead of $u^{n+1}$ in \eqref{eq:RK-final} as the new value after one time step.
If the entropy is just the energy $\eta(u) = \frac{1}{2} \norm{u}^2$, the choice
of $\gamma_n$ proposed in \cite{ketcheson2019relaxation} is such that
\begin{equation}
  \frac{1}{2} \norm{u^{n+1}_\gamma}^2 - \frac{1}{2} \norm{u^n}^2
  =
  \gamma_n \dt \sum_{i=1}^s b_i \scp{y_i}{f_i}.
\end{equation}
The new generalization to entropy stability proposed in this article is
to enforce the condition
\begin{equation}
  \eta(u^{n+1}_\gamma) - \eta(u^n)
  =
  \gamma_n \dt \sum_{i=1}^s b_i \scp{\eta'(y_i)}{f_i}
\end{equation}
by finding a root $\gamma_n$ of
\begin{equation}
\label{eq:r}
  r(\gamma)
  =
  \eta\biggl( u^n + \gamma \dt \sum_{i=1}^{s} b_{i} f_i \biggr) - \eta(u^n)
  - \gamma \dt \sum_{i=1}^s b_i \scp{\eta'(y_i)}{f_i}.
\end{equation}
Note that the direction
\begin{equation}
\label{eq:d}
  d^n := \sum_{i=1}^{s} b_{i} f_i
\end{equation}
and the estimate of the entropy change
\begin{equation}
\label{eq:e}
  e := \dt \sum_{i=1}^s b_i \scp{\eta'(y_i)}{f_i}
\end{equation}
can be computed on the fly during the computation of the Runge--Kutta method
and are not influenced by $\gamma_n$. Hence, existing low-storage implementations
can be used. In the end, finding a root of
$r(\gamma) = \eta(u^n + \gamma d) - \eta(u^n) - \gamma e$
is just a scalar root finding problem for the convex function $r$.

\begin{remark}
  If $f$ is a semi-discretization of a (hyperbolic) PDE with entropy $S$ and
  entropy variables $w(u) = S'(u)$ in the domain $\Omega$, \eqref{eq:r}
  corresponds to a discrete version of
  \begin{equation}
    r(\gamma)
    =
    \int_\Omega S( u^n + \gamma \dt \,d^n) \dif\Omega
    - \int_\Omega S(u^n) \dif\Omega
    - \gamma \dt \sum_{i=1}^s b_i \int_\Omega w_i \cdot f_i \dif\Omega,
  \end{equation}
  since the total entropy is $\eta(u) = \int_\Omega S(u) \dif\Omega$.
\end{remark}
If $f$ is a semi-discretization of a PDE and $\eta$ the global entropy, $r(\gamma=1)$
can be interpreted as global entropy production of the unmodified Runge--Kutta
method. Indeed, $\eta(u^{n+1}) - \eta(u^n)$ is the global entropy change and $e$
is the entropy change, which has the same sign as the true entropy time derivative
if the weights $b_i \geq 0$. Hence, $r$ will sometimes be called \emph{temporal entropy
production}. Thus, finding a root of $r$ yields a scheme that is entropy conservative
for conservative problems and entropy dissipative for dissipative problems.
This can be viewed as an extension
of \cite[Theorem~2.1]{ketcheson2019relaxation}, which dealt only with inner-product norms.
\begin{theorem}
  The method defined by \eqref{eq:RK-stages} \& \eqref{eq:RK-final-gamma}, where
  $\gamma_n$ is a root of \eqref{eq:r}, is conservative.
  If the weights $b_i$ are non-negative and $\gamma_n \geq 0$, then the method is
  dissipation preserving.
\end{theorem}

The new numerical solution $u^{n+1}_\gamma$ can be interpreted as an approximation
to either $u(t_n + \dt)$ (with scaled weights $\gamma_n b_i$) or to $u(t_n + \gamma_n \dt)$
(with scaled time step $\gamma_n \dt$).
As mentioned in \cite{ketcheson2019relaxation}, the given Runge--Kutta method determines
the direction $d$ and $\gamma_n$ can be interpreted as a relaxation parameter
determined by the requirement of preserving the evolution of $\eta$. Hence, the method defined by
\eqref{eq:RK-stages} \& \eqref{eq:RK-final-gamma} with the interpretation $u^{n+1}_\gamma
\approx u(t_n + \gamma_n \dt)$ is called a \emph{relaxation Runge--Kutta} (RRK)
method. The scheme using $u^{n+1}_\gamma \approx u(t_n + \dt)$ will be referred
to as an IDT method \cite{calvo2006preservation}.

\begin{remark}
\label{rem:bi-positive}
  Some well-known Runge--Kutta schemes do not satisfy the sufficient condition
  $b_i \ge 0$, $i \in \set{1, \dots, s}$.  For example,
  the classical fifth/fourth order pairs of Fehlberg and Dormand \& Prince
  have negative coefficients $b_5 < 0$.
\end{remark}

\subsection{Existence of a Solution}
\label{sec:existence}

Relaxation Runge--Kutta methods have been developed in
\cite{ketcheson2019relaxation} for the preservation of inner product
norms; in that setting $r(\gamma)$ is quadratic and its roots can be
explicitly computed. Here we deal instead with arbitrary functionals;
as we will see, new techniques are required.

Obviously, $r(0) = 0$ and $r$ is convex since the entropy $\eta$ is convex.
There is a positive root of $r$ if and only if $r(\gamma)$ is negative for small
$\gamma > 0$ and positive for large enough $\gamma > 0$.
\begin{lemma}
\label{lem:derivative-at-zero}
  Let a Runge--Kutta method be given with coefficients such that $\sum_{i=1}^s b_i a_{ij} > 0$
  and let $r(\gamma)$ be defined by \eqref{eq:r}.
  If $\eta''(u^n)(f_0,f_0) > 0$, then $r'(0) < 0$ for sufficiently small $\dt > 0$.
\end{lemma}
\begin{proof}
  By definition of $r$ \eqref{eq:r},
  \begin{equation}
  \begin{aligned}
    r'(0)
    &=
    \dt \sum_{i=1}^{s} b_{i} \scp{\eta'(u^n)}{f_i}
    - \dt \sum_{i=1}^s b_i \scp{\eta'(y_i)}{f_i}
    \\
    &=
    - \dt \sum_{i=1}^{s} b_{i}
    \int_0^1 \eta''\biggl( u^n + v \dt \sum_{k=1}^s a_{ik} f_k \biggr)
                   \biggl( f_i, \dt \sum_{j=1}^s a_{ij} f_j \biggr) \dif v.
  \end{aligned}
  \end{equation}
  Using Taylor expansions of $f_i, f_j = f_0 + \O(\dt)$,
  \begin{equation}
    r'(0)
    =
    - \dt^2 \sum_{i,j=1}^{s} b_{i} a_{ij}
    \int_0^1 \eta''\biggl( u^n + v \dt \sum_{k=1}^s a_{ik} f_k \biggr)(f_0, f_0) \dif v
    + \O(\dt)^3.
  \end{equation}
  Using the given assumptions, $r'(0) < 0$ for sufficiently small $\dt > 0$.
\end{proof}
\begin{remark}
  The assumption $\sum_{i=1}^s b_i a_{ij} > 0$ is satisfied for all (at least)
  second order accurate Runge--Kutta methods since
  $\sum_{i=1}^s b_i a_{ij} = \nicefrac{1}{2}$
  is a condition for second-order accuracy.
\end{remark}

\begin{lemma}
\label{lem:derivative-at-one}
  Let a Runge--Kutta method be given with coefficients satisfying
  $\sum_{i,j=1}^s b_i (a_{ij} - b_j) < 0$.
  If $\eta''(u^n)(f_0,f_0) > 0$, then $r'(1) > 0$ for sufficiently small $\dt > 0$.
\end{lemma}
\begin{proof}
  By definition of $r$ \eqref{eq:r},
  \begin{equation}
  \begin{aligned}
    r'(1)
    &=
    \dt \sum_{i=1}^{s} b_i \scp{\eta'(u^{n+1})}{f_i}
    - \dt \sum_{i=1}^s b_i \scp{\eta'(y_i)}{f_i}
    \\
    &=
    - \dt \sum_{i=1}^{s} b_{i}
    \int_0^1 \eta''\biggl( u^{n+1} + v \dt \sum_{k=1}^s (a_{ik} - b_k) f_k \biggr)
                   \biggl( f_i, \dt \sum_{j=1}^s (a_{ij} - b_j) f_j \biggr) \dif v.
  \end{aligned}
  \end{equation}
  Using Taylor expansions of $f_i, f_j = f_0 + \O(\dt)$,
  \begin{equation}
    r'(1)
    =
    - \dt^2 \sum_{i,j=1}^{s} b_{i} (a_{ij} - b_j)
    \int_0^1 \eta''\biggl( u^{n+1} + v \dt \sum_{k=1}^s (a_{ik} - b_k) f_k \biggr)(f_0, f_0) \dif v
    + \O(\dt)^3.
  \end{equation}
  Using the given assumptions, $r'(1) > 0$ for sufficiently small $\dt > 0$.
\end{proof}
\begin{remark}
  The assumption $\sum_{i,j=1}^s b_i (a_{ij} - b_j) < 0$ is satisfied for all
  (at least) second order accurate Runge--Kutta methods since
  $\sum_{i,j=1}^s b_i (a_{ij} - b_j) = \nicefrac{1}{2} - 1 = \nicefrac{-1}{2}$
  in that case.
\end{remark}

Together, these results establish the existence of a positive root of $r$.
\begin{theorem}
\label{thm:existence}
   Assume that the Runge--Kutta method satisfies $\sum_{i=1}^s b_i a_{ij} > 0$
   and $\sum_{i,j=1}^s b_i (a_{ij} - b_j) < 0$, which is true for all (at least)
   second order accurate schemes.
   If $\eta''(u^n)(f_0,f_0) > 0$, $r$ \eqref{eq:r} has a positive root for
   sufficiently small $\dt > 0$.
\end{theorem}
\begin{proof}
  Since $r(0) = 0$ and $r'(0) < 0$, $r(\gamma) < 0$ for small $\gamma > 0$.
  Because $r'(1) > 0$ and $r$ is convex, $r'$ is monotone. Hence, there must be
  a positive root of $r$.
\end{proof}

\begin{remark}
  The value $\eta''(u^n)(f_0,f_0)$ of the quadratic form $\eta''(u^n)$ is positive
  for a strictly convex entropy $\eta$ if $f_0 \neq 0$. If $f_0 = 0$ and the system
  is autonomous, every explicit Runge--Kutta method will yield a stationary solution.
  The results of Lemmas~\ref{lem:derivative-at-zero} \& \ref{lem:derivative-at-one}
  and hence of Theorem~\ref{thm:existence} still hold if we instead assume only that
  $\eta''(f_i,f_i) > 0$ for some intermediate stage $i$, since the Taylor series
  can be expanded around that value.
\end{remark}

\begin{remark}
\label{rem:shape-of-r-1}
  The proof of Theorem~\ref{thm:existence} reveals another property of $r$:
  the temporal entropy dissipation.  Since $r$ is convex, there are exactly two
  distinct roots of $r$, namely zero and the desired positive root $\gamma_n$ (if the
  assumptions of Theorem~\ref{thm:existence} are satisfied). Additionally,
  $r(\gamma) \to \infty$ for $\gamma \to \pm\infty$. Therefore, choosing a
  value of $\gamma > 0$ smaller than the positive root of $r$ results in some
  additional temporal entropy dissipation, because $r(\gamma) < 0$ in that case.
\end{remark}

\subsection{Accuracy}
\label{sec:accuracy}

At first glance, the method described above seems to be not even consistent,
since $\gamma_n \sum_j b_j = \gamma_n \ne 1$ in general.
Nevertheless, an RRK scheme is of at least the same order of accuracy
as the RK scheme it is based on. In order to prove this, we obtain several
results, which will be combined and are also interesting on their own.
Readers who are interested only in the statement of the main accuracy
result can skip these parts and continue with Theorem~\ref{thm:accuracy}
and Remark~\ref{rem:shape-of-r-2}.

The following result has been
obtained in \cite[Theorem~2.4]{ketcheson2019relaxation}.
\begin{theorem}
\label{thm:accuracy-basic}
  Let the given Runge--Kutta method be of order $p$. Consider the IDT/RRK method
  defined by \eqref{eq:RK-stages} \& \eqref{eq:RK-final-gamma} and suppose that
  $\gamma_n = 1 + \O(\dt^{p-1})$.
  \begin{enumerate}
    \item
    The IDT method interpreting $u^{n+1}_\gamma \approx u(t_n + \dt)$
    has order $p-1$.

    \item
    The relaxation method interpreting $u^{n+1}_\gamma \approx u(t_n + \gamma_n \dt)$
    has order $p$.
  \end{enumerate}
\end{theorem}

Using $u^n$ as initial value for $u$ at $t_n$, a Runge--Kutta method with order
of accuracy $p$ yields
\begin{equation}
\label{eq:first-step-to-accuracy}
\begin{aligned}
  &\phantom{=\eta}
  \eta(u^{n+1}) - \eta(u^n)
  =
  \eta(u(t_n + \dt)) - \eta(u^n) + \O(\dt^{p+1})
  \\
  &=
  \int_{t_n}^{t_n + \dt} \scp{\eta'(u(t))}{f(t, u(t))} \dif t + \O(\dt^{p+1})
  \\
  &=
  \dt \sum_{i=1}^s b_i \scp{\eta'(u(t_n + c_i \dt))}{f(t_n + c_i \dt, u(t_n + c_i \dt))}
  + \O(\dt^{p+1})
\end{aligned}
\end{equation}
because of the required accuracy as a quadrature rule. Although the stage values
$y_i$ are not necessarily high-order approximations of $u(t_n + c_i \dt)$, the
Runge--Kutta order conditions guarantee
\begin{equation}
  \sum_{i=1}^s b_i f(t_n + c_i \dt, y_i)
  =
  \sum_{i=1}^s b_i f(t_n + c_i \dt, u(t_n + c_i \dt)) + \O(\dt^p).
\end{equation}
Hence, it is interesting to know whether $f$ can be replaced by any smooth
function in this equation.
\begin{theorem}
\label{thm:integration-accuracy}
  Let $W$ be a Banach space, $\psi\colon [0,T] \times \H \to W$ a smooth function, and
  $b_i, c_i$ coefficients of a Runge--Kutta method of order $p$.  Then
  \begin{equation}
    \sum_{i=1}^s b_i \psi(t_n + c_i \dt, y_i)
    =
    \sum_{i=1}^s b_i \psi(t_n + c_i \dt, u(t_n + c_i \dt)) + \O(\dt^p).
  \end{equation}
\end{theorem}
\begin{corollary}
\label{cor:accuracy}
  If $\eta$ is smooth and the given Runge--Kutta method is $p$-th order accurate,
  $r(\gamma=1) = \O(\dt^{p+1})$.
\end{corollary}
\begin{proof}[Proof of Corollary~\ref{cor:accuracy}]
  Apply Theorem~\ref{thm:integration-accuracy} to $\psi(t,u) = \scp{\eta'(u)}{f(t, u)}$
  and use \eqref{eq:first-step-to-accuracy}, resulting in
  \begin{equation}
    \eta(u^{n+1}) - \eta(u^n)
    =
    \dt \sum_{i=1}^s b_i \scp{\eta'(y_i)}{f(t_n + c_i \dt, y_i)}
  + \O(\dt^{p+1}).
  \end{equation}
\end{proof}
\begin{proof}[Proof of Theorem~\ref{thm:integration-accuracy}]
  Consider $\phi(t) = \int_{t_n}^t \psi(\tau, u(\tau)) \dif \tau$.
  Applying the Runge--Kutta method to the extended ODE (with a slight abuse of
  notation)
  \begin{equation}
  \label{eq:extended-ode}
    \od{}{t} \underbrace{\begin{pmatrix} \phi(t) \\ u(t) \end{pmatrix}}_{= x(t)}
    =
    \begin{pmatrix} \psi(t, u(t)) \\ f(t, u(t)) \end{pmatrix},
    \; t \in (t_n,T),
    \qquad
    \begin{pmatrix} \phi(t_n) \\ u(t_n) \end{pmatrix}
    =
    \begin{pmatrix} 0 \\ u^n \end{pmatrix},
  \end{equation}
  yields the same stage values $y_i$ for the second component $u$ of $x$. Since
  the method is $p$-th order accurate,
  \begin{equation}
  \label{eq:integration-accuracy-1}
    \dt \sum_{i=1}^s b_i \psi(t_n + c_i \dt, y_i)
    =
    \phi^{n+1}
    =
    \phi(t_n + \dt) + \O(\dt^{p+1}).
  \end{equation}
  Additionally,
  \begin{equation}
  \label{eq:integration-accuracy-2}
    \phi(t_n + \dt)
    =
    \int_{t_n}^{t_n + \dt} \psi(t, u(t)) \dif t
    =
    \dt \sum_{i=1}^s b_i \psi(t_n + c_i \dt, u(t_n + c_i \dt))
    + \O(\dt^{p+1}).
  \end{equation}
  Combining \eqref{eq:integration-accuracy-1} and \eqref{eq:integration-accuracy-2}
  yields the desired result.
\end{proof}
\begin{remark}
  Theorem~\ref{thm:integration-accuracy} can be seen as a superconvergence
  result for integrals evaluated using the quadrature rule associated with a
  Runge--Kutta method.
  It extends a related result of \cite[Lemma~4]{ketcheson2019relaxation}
  in two ways. Firstly, general functionals instead of the energy are
  considered. Secondly, the proof is simplified and does not rely on
  extensive computations involving the theory of Butcher series.
\end{remark}

\begin{theorem}
\label{thm:accuracy}
  Assume that the conditions of Theorem~\ref{thm:existence} are satisfied. Hence,
  there exists a unique positive root $\gamma_n$ of $r$ \eqref{eq:r}.
  Consider the IDT/RRK method defined by \eqref{eq:RK-stages} \&
  \eqref{eq:RK-final-gamma} and suppose that the given Runge--Kutta method is
  $p$-th order accurate.
  \begin{enumerate}
    \item
    The IDT method interpreting $u^{n+1}_\gamma \approx u(t_n + \dt)$
    has order $p-1$.

    \item
    The relaxation method interpreting $u^{n+1}_\gamma \approx u(t_n + \gamma_n \dt)$
    has order $p$.
  \end{enumerate}
\end{theorem}
\begin{proof}
  Because of Corollary~\ref{cor:accuracy}, $r(1) = \O(\dt^{p+1})$.
  As can be seen in the proof of Lemma~\ref{lem:derivative-at-one},
  $r'(1) = c \dt^2 + \O(\dt^3)$, where $c > 0$.
  Hence, there is a root $\gamma_n = 1 + \O(\dt^{p-1})$ of $r$ \eqref{eq:r}.
  Applying Theorem~\ref{thm:accuracy-basic} yields the desired accuracy result.
\end{proof}

\begin{remark}
\label{rem:shape-of-r-2}
  As an extension of Remark~\ref{rem:shape-of-r-1}, the behavior of the temporal
  entropy dissipation $r$ \eqref{eq:r} can be described as follows for sufficiently
  small $\dt$ if the assumptions of Theorem~\ref{thm:existence} are satisfied:
  Firstly, $r(0) = 0$, $r(1) = \O(\dt^{p+1}) \approx 0$, and there is a
  unique $0 < \gamma_n = 1 + \O(\dt^{p-1})$ such that $r(\gamma_n) = 0$. Between
  zero and this root of $r$, the values of $r$ are negative, i.e. additional
  entropy dissipation is introduced in that region. Outside of the bounded
  interval given by zero and $\gamma_n$, $r$ is positive and the time
  integration scheme produces entropy. Additionally, $r(\gamma) \to \infty$ for
  $\gamma \to \pm\infty$. Finally, $r$ is convex and looks approximately similar
  to a parabola with the same roots for sufficiently small $\dt > 0$.
  See Figure~\ref{fig:cons_exp_entropy__r_over_gamma} for a typical plot of $r(\gamma)$.
\end{remark}

\begin{remark}
\label{rem:higher-order-possible}
  Theorem~\ref{thm:accuracy} gives a guaranteed minimal order of accuracy.
  For some specific problems and schemes, the resulting order of accuracy
  can be even greater. For example, applying the classical third order,
  three stage method of Heun to the harmonic oscillator
  \begin{equation}
    u_1'(t) = -u_2(t), \quad u_2'(t) = u_1(t),
  \end{equation}
  with entropy (energy) $\eta(u) = \norm{u}^2  / 2$,
  it can be shown that with relaxation the rate of convergence is
  \emph{fourth order}.
  The same result holds true for a nonlinear oscillator given by
  \begin{equation}
    u_1'(t) = -\norm{u}^2 u_2(t), \quad u_2'(t) = \norm{u}^2 u_1(t),
  \end{equation}
  and the same entropy $\eta$.
\end{remark}

\subsection{Additional Properties and Generalizations}
\label{sec:properties}

As described in \cite{ketcheson2019relaxation}, relaxation RK methods still conserve
linear invariants, although $\gamma_n$ is determined in a nonlinear way. Such linear
invariants are e.g. the total mass for a semi-discretization of a hyperbolic
conservation law in a periodic domain.

Another desirable stability property of numerical time integration schemes is
the preservation of convex stability properties that hold for the explicit
Euler method. Such schemes are called \emph{strong stability preserving} (SSP),
as described in the monograph \cite{gottlieb2011strong} and references cited
therein. It has been shown in \cite[Section~3]{ketcheson2019relaxation} that
the relaxation modification of many SSP methods retains the same SSP property
of the original method as long as $\gamma_n$ deviates not too much from unity.

If there are several convex quantities $\eta_i$ which do not necessarily have to
be conserved but might also be dissipated, one could compute a relaxation factor
$\gamma_{n,i}$ for every $\eta_i$ and choose $\gamma_n = \min_i \gamma_{n,i}$. The resulting
scheme will dissipate every entropy (if $b_i \geq 0$) because of the general
shape of the temporal entropy dissipation $r$, cf. Remark~\ref{rem:shape-of-r-2}.

If concave quantities (which shall typically increase) are of interest, they
can be treated in the same framework using a sign change of $\eta$. If general
functions $\eta$ without any convexity/concavity assumptions are of interest,
relaxation and IDT methods can still be applied.
\begin{proposition}
\label{pro:general-functions}
  Suppose that the given Runge--Kutta method is $p$-th order accurate with $p \geq 2$.
  If $\scp{\eta'(u^{n+1})}{d^n / \norm{d^n}}
  = B(u^n) \dt + \O(\dt^2)$ with $B(u^n) \neq 0$,
  then $r$ \eqref{eq:r} has a positive root $\gamma_n = 1 + \O(\dt^{p-1})$.
  If this root is used to define IDT/RRK methods by \eqref{eq:RK-stages} \&
  \eqref{eq:RK-final-gamma}, then:
  \begin{enumerate}
    \item
    The IDT method interpreting $u^{n+1}_\gamma \approx u(t_n + \dt)$
    has order $p-1$.

    \item
    The relaxation method interpreting $u^{n+1}_\gamma \approx u(t_n + \gamma_n \dt)$
    has order $p$.
  \end{enumerate}
\end{proposition}
\begin{proof}
  The proof of \cite[Theorem~2]{calvo2010projection} using the implicit function
  theorem can be adapted to this setting; the normalized search direction considered
  there is $w = d^n / \norm{d^n}$ and the projected value is
  $u^{n+1}_\gamma = u^{n+1} + (1-\gamma_n) \dt \, d^n = u^{n+1} + \lambda_n w$,
  i.e. the step parameters are related via $\gamma_n = 1 + \lambda_n / \norm{\dt \, d^n}$.
  Since there is a solution $\lambda_n = \O(\dt^{p})$ and $\dt \, d^n = \dt \sum_{i=1}^s b_i f_i$
  scales as $\dt$, there is a solution $\gamma_n = 1 + \O(\dt^{p-1})$. Applying
  Theorem~\ref{thm:accuracy-basic} yields the desired results.
\end{proof}

\begin{remark}
  While Proposition~\ref{pro:general-functions} can be applied to general functions
  $\eta$, the detailed existence and accuracy results developed in the previous sections
  reveal more properties in the convex case and provide additional insights. These
  additional properties (such as the general shape of $r$, possible entropy
  dissipation by smaller values of $\gamma_n$) are useful for applications and
  root finding procedures.
\end{remark}

\subsection{Implementation}
\label{sec:implementation}

For a given Runge--Kutta method with coefficients $a_{ij}$, $b_i$, the relaxation method
defined by \eqref{eq:RK-stages} \& \eqref{eq:RK-final-gamma} requires additionally
only the solution of a scalar equation, which can be done effectively using standard
methods. The derivative of $r$ is
\begin{equation}
  r'(\gamma) = \scp{\eta'(u^n + \gamma \dt \, d^n)}{\dt \, d^n} - e,
\end{equation}
where the direction $\dt \, d^n$ and the estimate $e$ are defined
as in \eqref{eq:r} and \eqref{eq:e}, respectively.

For most of the numerical experiments presented below, \texttt{scipy.optimize.brentq}
(using Brent’s method \cite[Chapters~3--4]{brent1973algorithms})
or \texttt{scipy.optimize.root} with \texttt{method='lm'} (using a modification of the Levenberg-Marquardt algorithm as implemented in MINPACK \cite{more1980user})
from SciPy \cite{scipy} have been used. In most cases, Brent’s method is more
efficient.  For the first step, $\gamma = 1$ is a good initial guess;
cf. Section~\ref{sec:accuracy}.  In subsequent steps the previous value of
$\gamma$ is chosen as initial guess, since $\gamma$ changes only slightly from
step to step. Implementations used for the numerical examples up to
section~\ref{sec:other-eqs} are provided in
\cite{ranocha2019relaxationRepository}.

In particular for any convex entropy $\eta$, standard results of numerical
analysis guarantee that Newton's method converges if the conditions of the
existence and accuracy theorems are satisfied
\cite[Theorem~1.9]{suli2003introduction}. Optimized implementations that
are robust and efficient for both small (ODE) and large (PDE) problems
are left for future research.

\section{Numerical Examples\label{sec:numerical}}

The following Runge--Kutta methods with weights $b_i \geq 0$ will be used in the
numerical experiments.  The value of $\dt$ is fixed in each test and embedded
error estimators are not used.
\begin{itemize}
  \item
  SSPRK(2,2): Two stage, second order SSP method of \cite{shu1988efficient}.

  \item
  SSPRK(3,3): Three stage, third order SSP method of \cite{shu1988efficient}.

  \item
  SSPRK(10,4): Ten stage, fourth order SSP method of \cite{ketcheson2008highly}.

  \item
  RK(4,4): Classical four stage, fourth order method.

  \item
  BSRK(8,5): Eight stage, fifth order method of \cite{bogacki1996efficient}.

  \item
  VRK(9,6): Nine stage, sixth order method of the family developed
  in \cite{verner1978explicit}\footnote{The coefficients are taken from \url{http://people.math.sfu.ca/~jverner/RKV65.IIIXb.Robust.00010102836.081204.CoeffsOnlyFLOAT}
  at 2019-04-27.}.

  \item
  VRK(13,8): Thirteen stage, eight order method of the family developed
  in \cite{verner1978explicit}\footnote{The coefficients are taken from \url{http://people.math.sfu.ca/~jverner/RKV87.IIa.Robust.00000754677.081208.CoeffsOnlyFLOAT}
  at 2019-04-27.}.
\end{itemize}

\subsection{Conserved Exponential Entropy}
\label{sec:cons_exp_entropy}

Consider the system
\begin{equation}
\label{eq:cons_exp_entropy}
  \od{}{t} \begin{pmatrix} u_1(t) \\ u_2(t) \end{pmatrix}
  =
  \begin{pmatrix} -\exp(u_2(t)) \\ \exp(u_1(t)) \end{pmatrix},
  \quad
  u^0
  =
  \begin{pmatrix} 1 \\ 0.5 \end{pmatrix},
\end{equation}
with exponential entropy
\begin{equation}
\label{eq:exp-entropy-2}
  \eta(u)
  =
  \exp(u_1) + \exp(u_2),
  \quad
  \eta'(u)
  =
  \begin{pmatrix}
    \exp(u_1) \\
    \exp(u_2)
  \end{pmatrix},
\end{equation}
which is conserved for the analytical solution
\begin{equation}
  u(t)
  =
  \biggl(
    \log\biggl( \frac{\e^{(\sqrt{\e} + \e) t} (\sqrt{\e} + \e)}
                     {\sqrt{\e} + \e^{(\sqrt{\e} + \e) t}} \biggr),\,
    \log\Bigl( \e + \e^{3/2} \Bigr) - \log\Bigl( \sqrt{\e} + \e^{(\sqrt{\e} + \e) t} \Bigr)
  \biggr)^T.
\end{equation}

The shape of $r(\gamma)$ for the first time step using SSPRK(3,3) is shown
in Figure~\ref{fig:cons_exp_entropy__r_over_gamma}. In accordance with the
description given in Remark~\ref{rem:shape-of-r-2}, $r(0) = 0$, $r(1) \approx 0$,
$r$ is negative between its roots and positive outside of this interval.
The order of accuracy $r(1) = \O(\dt^{p+1})$ guaranteed by Corollary~\ref{cor:accuracy}
is obtained for the methods shown in Figure~\ref{fig:cons_exp_entropy__r_at_one}.

\begin{figure}
\centering
  \begin{subfigure}[b]{0.4\textwidth}
    \centering
    \includegraphics[width=\textwidth]{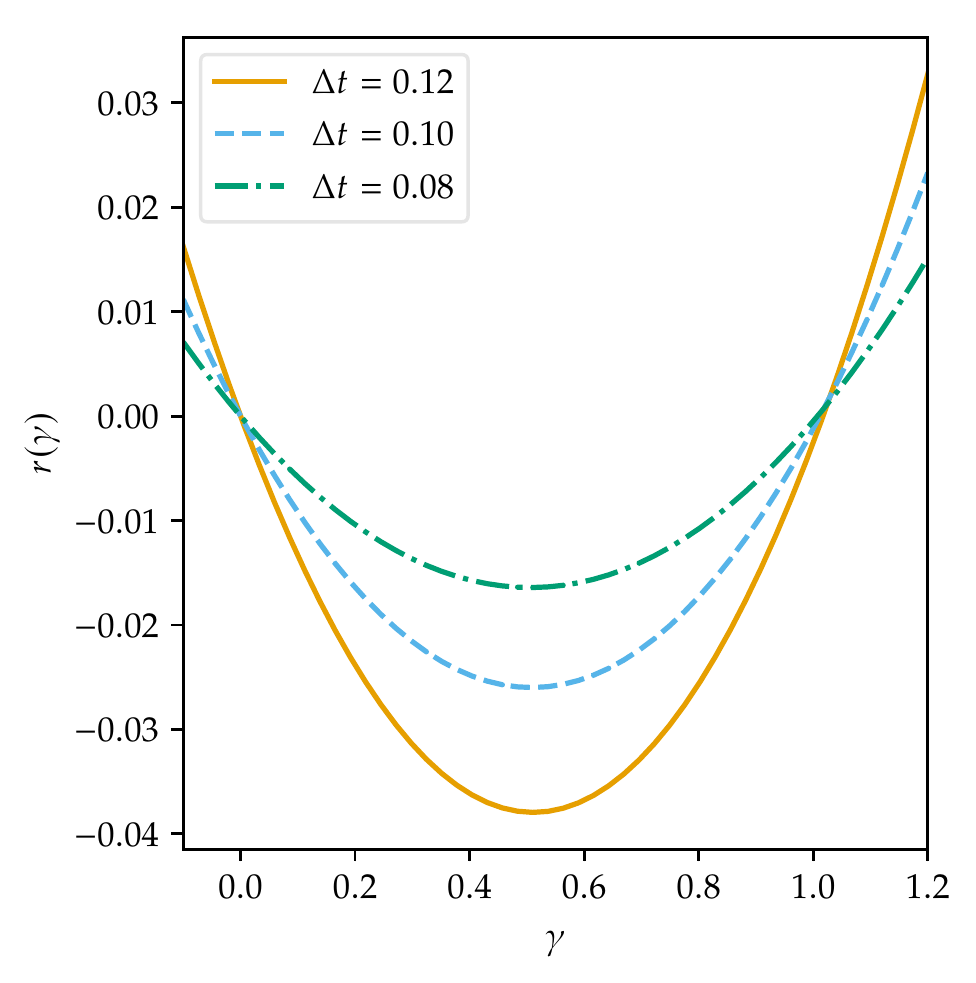}
    \caption{$r(\gamma)$ for SSPRK(3,3).}
    \label{fig:cons_exp_entropy__r_over_gamma}
  \end{subfigure}%
  \begin{subfigure}[b]{0.6\textwidth}
    \centering
    \includegraphics[width=\textwidth]{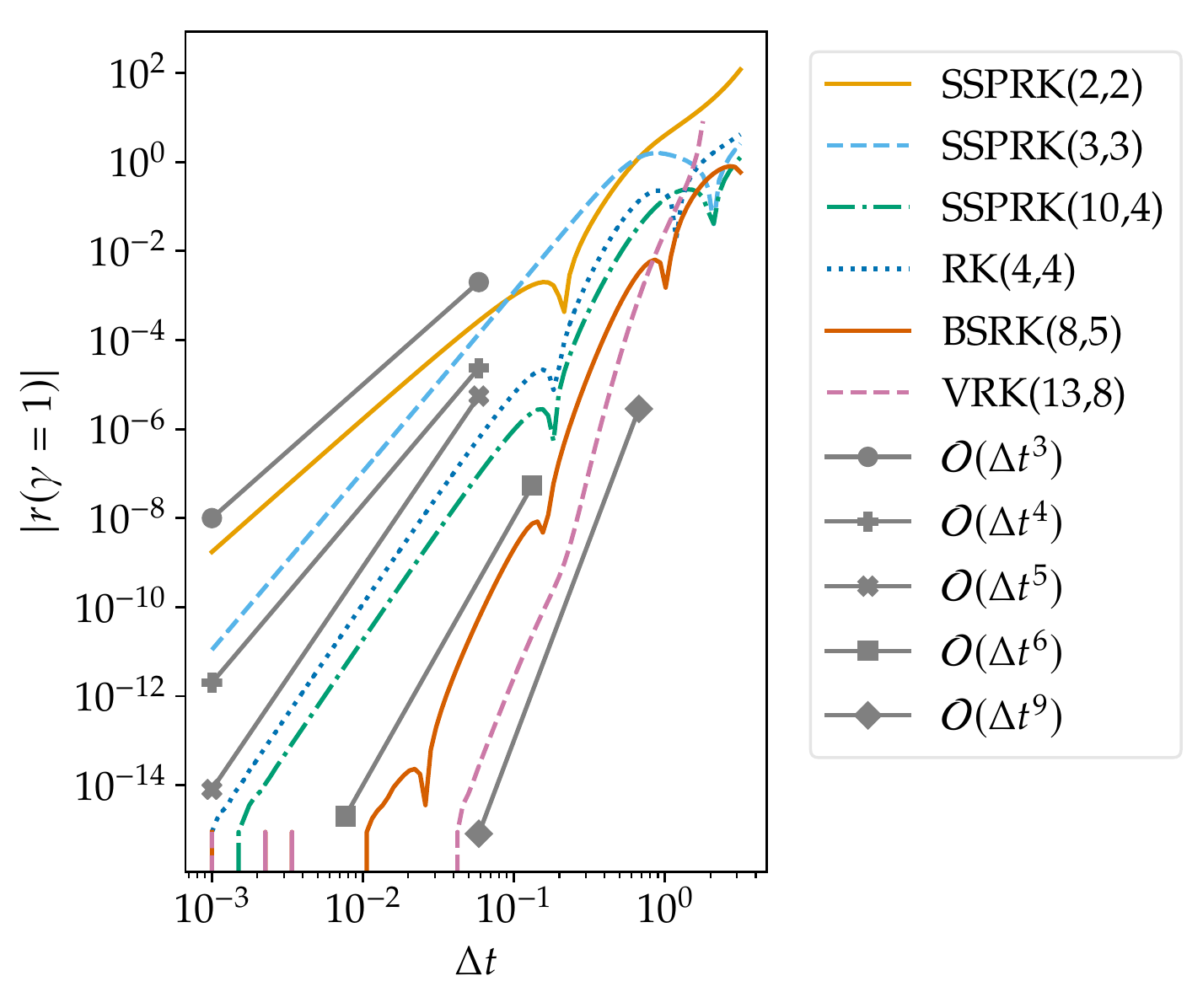}
    \caption{$r(\gamma=1)$ for some Runge--Kutta methods.}
    \label{fig:cons_exp_entropy__r_at_one}
  \end{subfigure}%
  \caption{Numerical results for the temporal entropy production $r$ \eqref{eq:r}
           at the first time step for the entropy conservative ODE
           \eqref{eq:cons_exp_entropy}.}
\end{figure}

Results of a convergence study in this setup are shown in
Figure~\ref{fig:cons_exp_entropy__convergence}. The unmodified and relaxation schemes
($u^{n+1}_\gamma \approx u(t^{n} + \gamma_n \dt)$) converge with the
expected order of accuracy $p$, in accordance with Theorem~\ref{thm:accuracy}.
The IDT methods ($u^{n+1}_\gamma \approx u(t^{n} + \dt)$) yield a
reduced order of convergence according to Theorem~\ref{thm:accuracy}. Moreover,
they are far more sensitive to variations of the nonlinear solvers (algorithms,
tolerances, and other related parameters) and show
serious convergence issues for small time steps in this case, as can be seen in
Figure~\ref{fig:cons_exp_entropy__convergence_relaxed_olddt}. Hence, the relaxation
schemes are far superior in this case.

\begin{figure}
\centering
  \begin{subfigure}[b]{\textwidth}
    \centering
    \includegraphics[width=\textwidth]{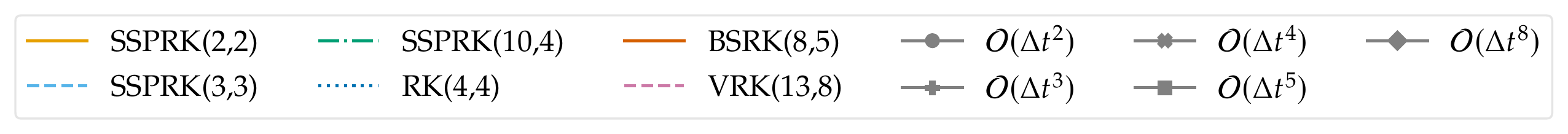}
  \end{subfigure}%
  \\
  \begin{subfigure}[b]{0.33\textwidth}
    \centering
    \includegraphics[width=\textwidth]{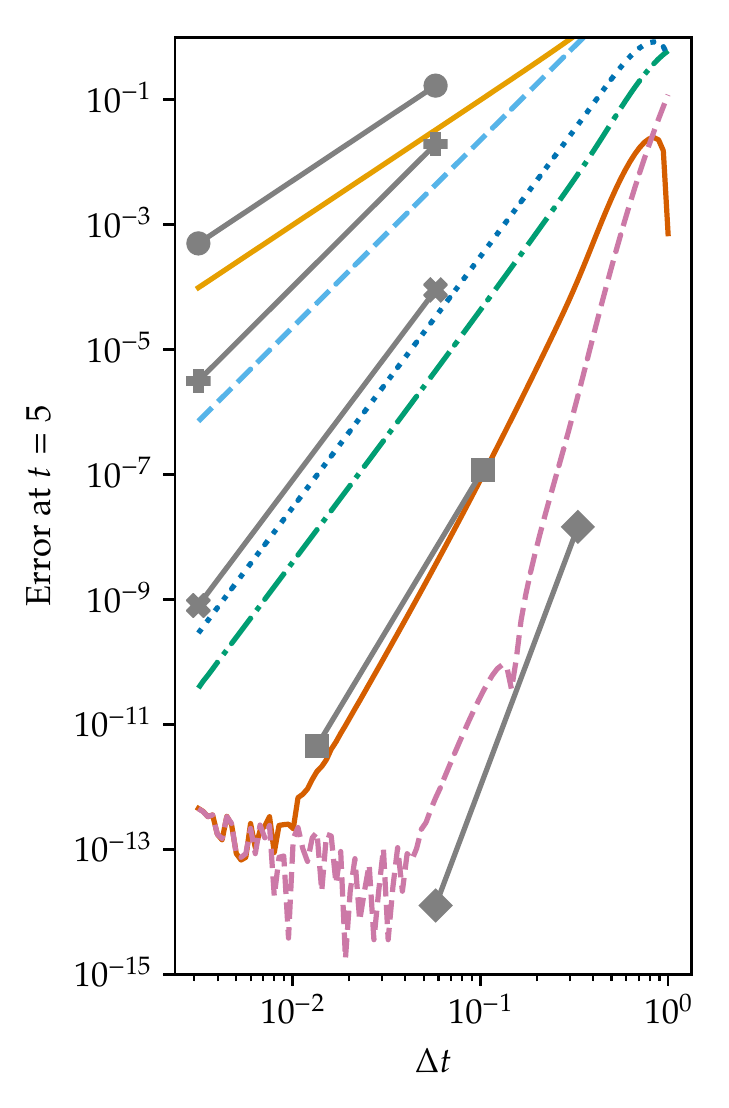}
    \caption{Unmodified methods.}
    \label{fig:cons_exp_entropy__convergence_standard}
  \end{subfigure}%
  \begin{subfigure}[b]{0.33\textwidth}
    \centering
    \includegraphics[width=\textwidth]{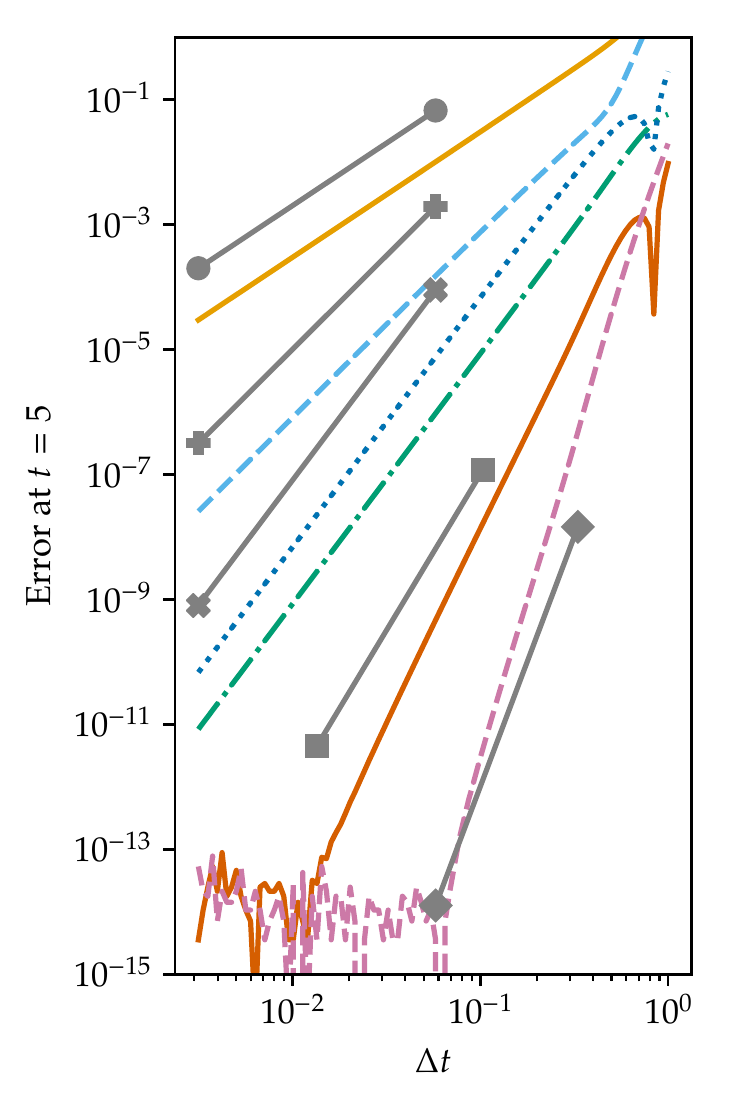}
    \caption{Relaxation methods.}
    \label{fig:cons_exp_entropy__convergence_relaxed}
  \end{subfigure}%
  \begin{subfigure}[b]{0.33\textwidth}
    \centering
    \includegraphics[width=\textwidth]{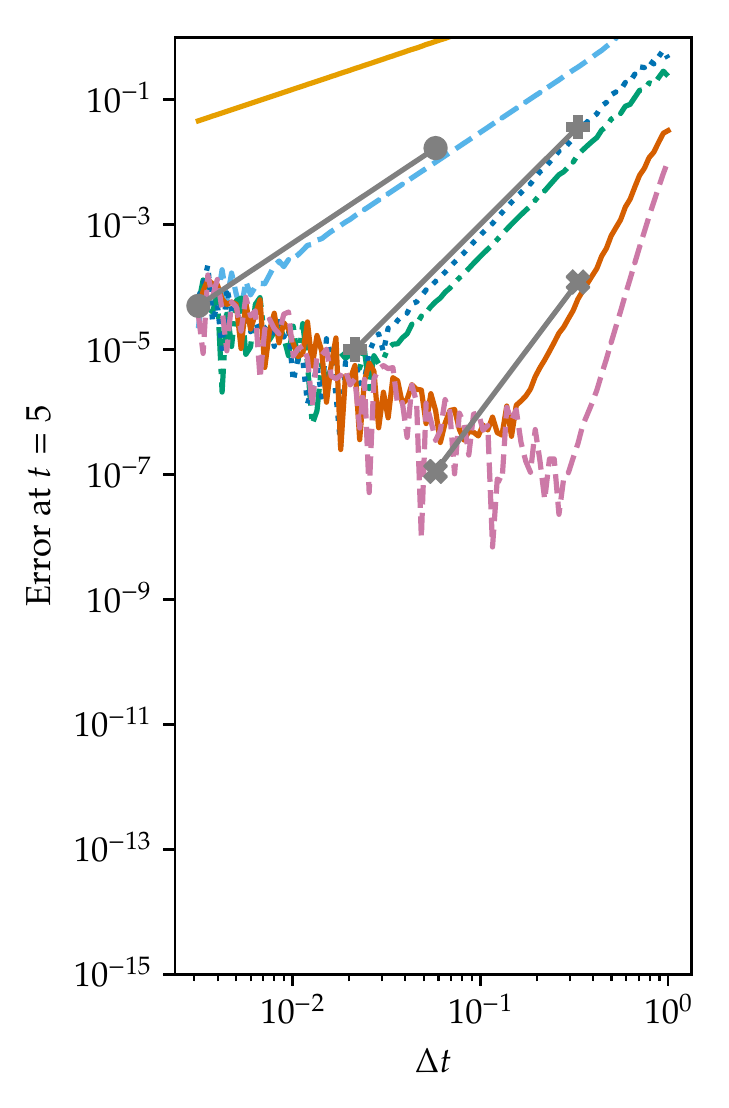}
    \caption{IDT methods.}
    \label{fig:cons_exp_entropy__convergence_relaxed_olddt}
  \end{subfigure}%
  \caption{Convergence study for the entropy conservative ODE \eqref{eq:cons_exp_entropy}
           with unmodified methods, RRK schemes ($u^{n+1}_\gamma
           \approx u(t^{n} + \gamma_n \dt)$), and IDT methods ($u^{n+1}_\gamma
           \approx u(t^{n} + \dt)$).}
  \label{fig:cons_exp_entropy__convergence}
\end{figure}

\subsection{Dissipated Exponential Entropy}
\label{sec:diss_exp_entropy}

Consider the ODE
\begin{equation}
\label{eq:diss_exp_entropy}
  \od{}{t} u(t)
  =
  -\exp(u(t)),
  \quad
  u^0
  =
  0.5,
\end{equation}
with exponential entropy $\eta(u) = \exp(u)$,
which is dissipated for the analytical solution
\begin{equation}
  u(t)
  =
  -\log\Bigl( \e^{-1/2} + t \Bigr).
\end{equation}

The shape of $r$ and the convergence behavior of $r(1) \to 0$ as $\dt \to 0$ are
very similar to the ones of Section~\ref{sec:cons_exp_entropy} and are therefore
not shown in detail. However, the dissipative system \eqref{eq:diss_exp_entropy}
results in a better convergence behavior of the modified schemes: They
depend less on the nonlinear solvers and there are less problems
for small $\dt$. Nevertheless, the order of convergence using the RRK schemes
is still better than for the IDT methods, as explained by Theorem~\ref{thm:accuracy}.

\begin{figure}
\centering
  \begin{subfigure}[b]{\textwidth}
    \centering
    \includegraphics[width=\textwidth]{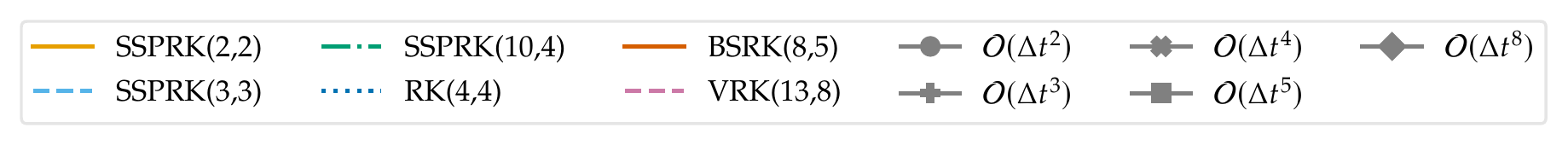}
  \end{subfigure}%
  \\
  \begin{subfigure}[b]{0.33\textwidth}
    \centering
    \includegraphics[width=\textwidth]{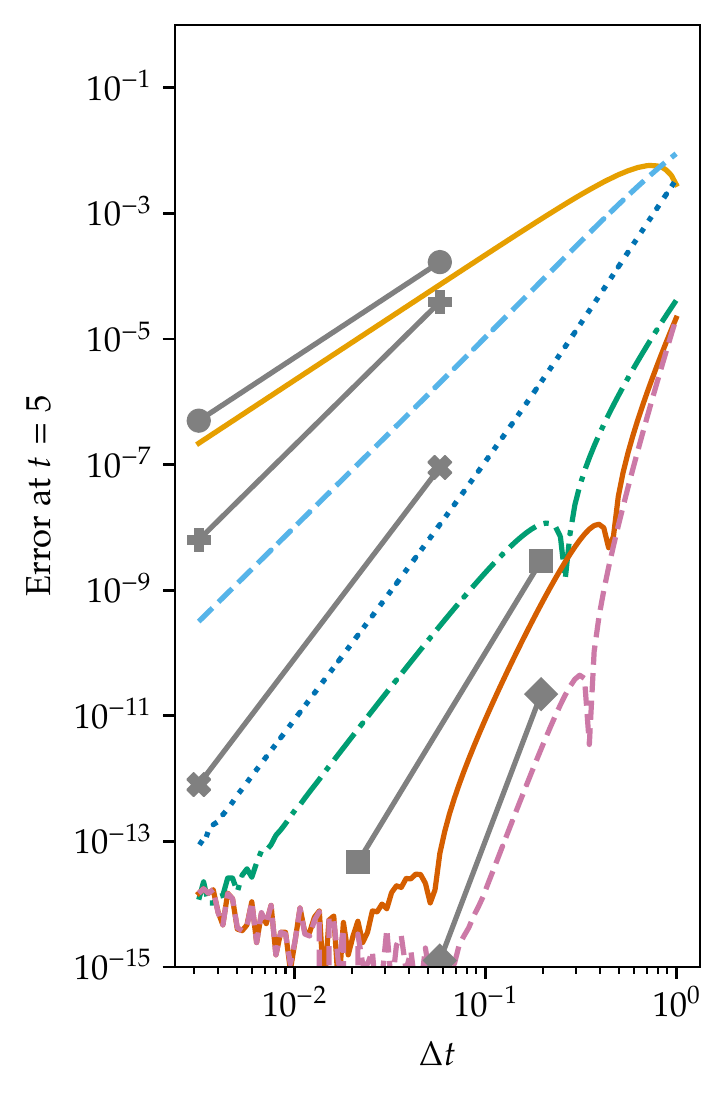}
    \caption{Unmodified methods.}
    \label{fig:diss_exp_entropy__convergence_standard}
  \end{subfigure}%
  \begin{subfigure}[b]{0.33\textwidth}
    \centering
    \includegraphics[width=\textwidth]{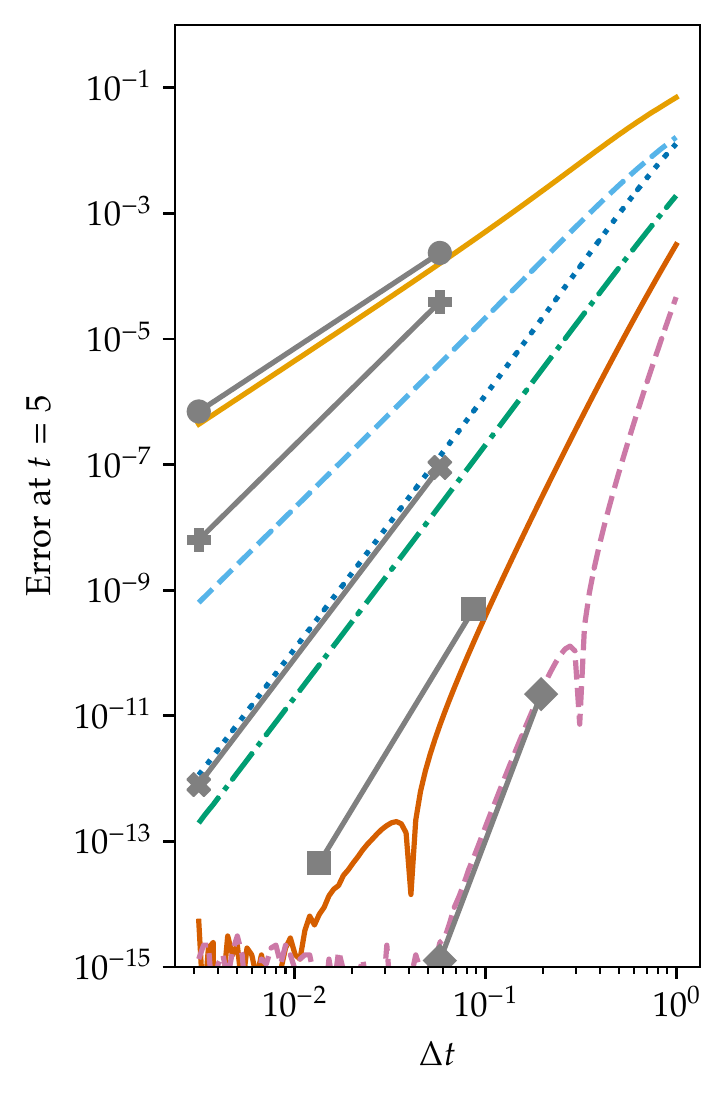}
    \caption{Relaxation methods.}
    \label{fig:diss_exp_entropy__convergence_relaxed}
  \end{subfigure}%
  \begin{subfigure}[b]{0.33\textwidth}
    \centering
    \includegraphics[width=\textwidth]{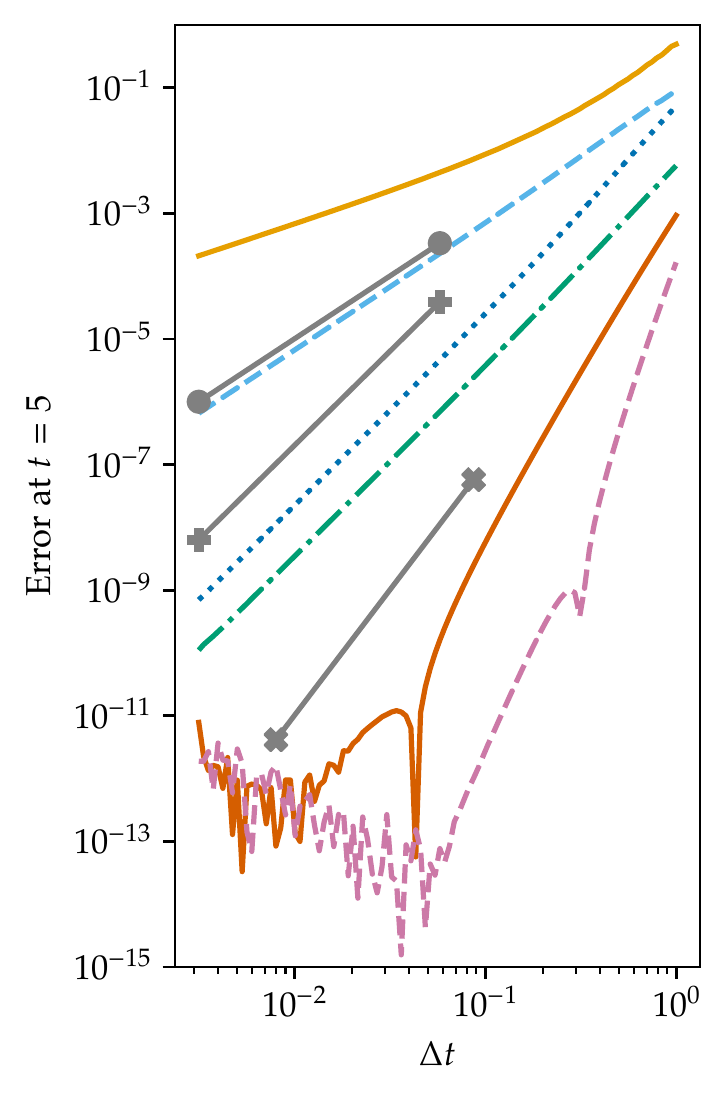}
    \caption{IDT methods.}
    \label{fig:diss_exp_entropy__convergence_relaxed_olddt}
  \end{subfigure}%
  \caption{Convergence study for the entropy dissipative ODE \eqref{eq:diss_exp_entropy}
           with unmodified methods, RRK schemes ($u^{n+1}_\gamma
           \approx u(t^{n} + \gamma_n \dt)$), and IDT methods ($u^{n+1}_\gamma
           \approx u(t^{n} + \dt)$).}
  \label{fig:diss_exp_entropy__convergence}
\end{figure}

\subsection{Nonlinear Pendulum}
\label{sec:pendulum}

Consider the system
\begin{equation}
\label{eq:pendulum}
  \od{}{t} \begin{pmatrix} u_1(t) \\ u_2(t) \end{pmatrix}
  =
  \begin{pmatrix} -\sin(u_2(t)) \\ u_1(t) \end{pmatrix},
  \quad
  u^0
  =
  \begin{pmatrix} 1.5 \\ 1 \end{pmatrix},
\end{equation}
with non-quadratic energy
\begin{equation}
\label{eq:pendulum_energy}
  \eta(u)
  =
  \frac{1}{2} u_1^2 - \cos(u_2),
  \quad
  \eta'(u)
  =
  \begin{pmatrix}
    u_1 \\
    \sin(u_2)
  \end{pmatrix},
\end{equation}
which is conserved for all $u$. Further, this entropy is convex for
all $u_1$ and $\abs{u_2} < \nicefrac{\pi}{2}$. Note that the border
of the convex region is crossed for this initial condition, since
$\abs{u_2}$ becomes larger than $\nicefrac{\pi}{2}$.

\begin{figure}
\centering
  \begin{subfigure}[b]{\textwidth}
    \centering
    \includegraphics[width=\textwidth]{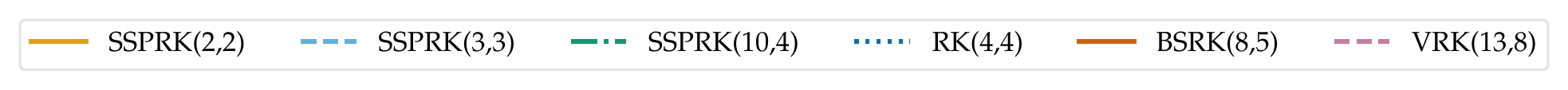}
  \end{subfigure}%
  \\
  \begin{subfigure}[b]{0.33\textwidth}
    \centering
    \includegraphics[width=\textwidth]{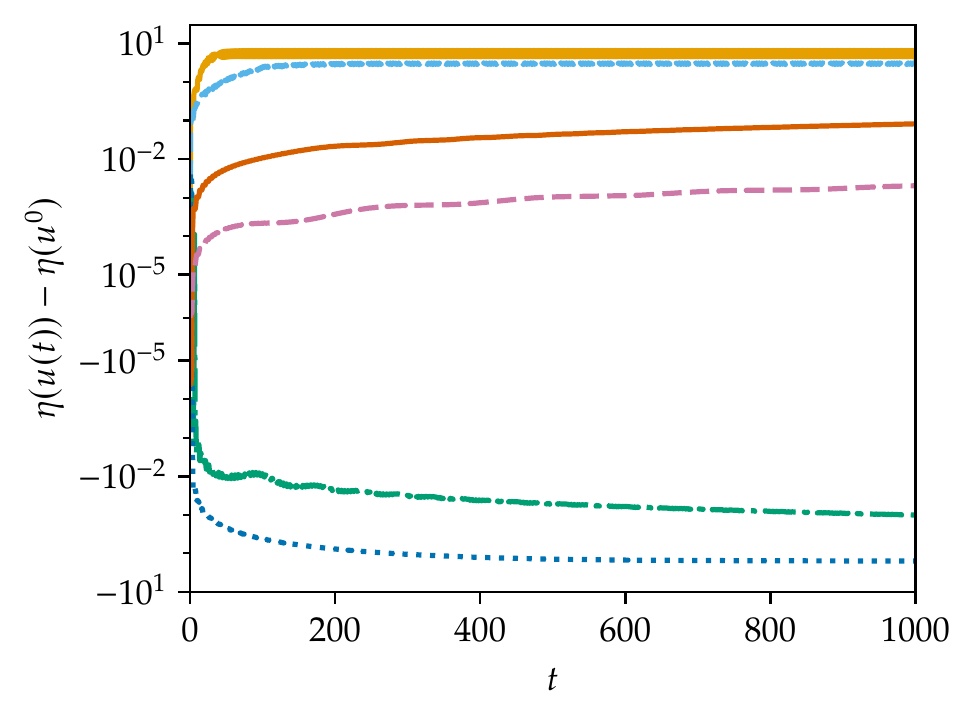}
    \caption{Unmodified methods.}
    \label{fig:pendulum__energy_standard}
  \end{subfigure}%
  \begin{subfigure}[b]{0.33\textwidth}
    \centering
    \includegraphics[width=\textwidth]{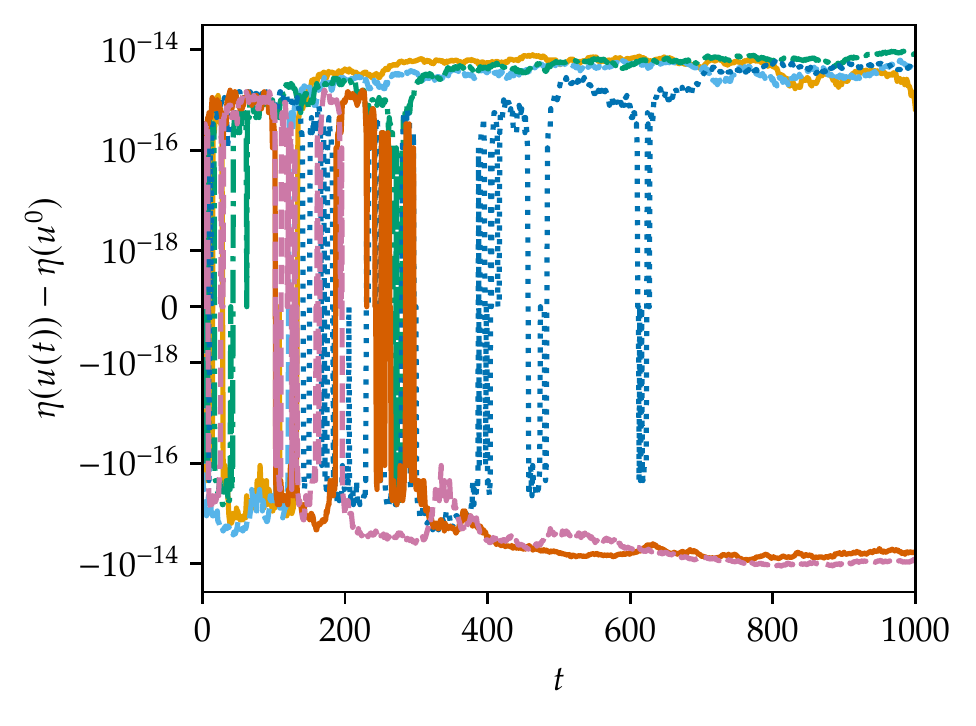}
    \caption{Relaxation methods.}
    \label{fig:pendulum__energy_relaxed}
  \end{subfigure}%
  \begin{subfigure}[b]{0.33\textwidth}
    \centering
    \includegraphics[width=\textwidth]{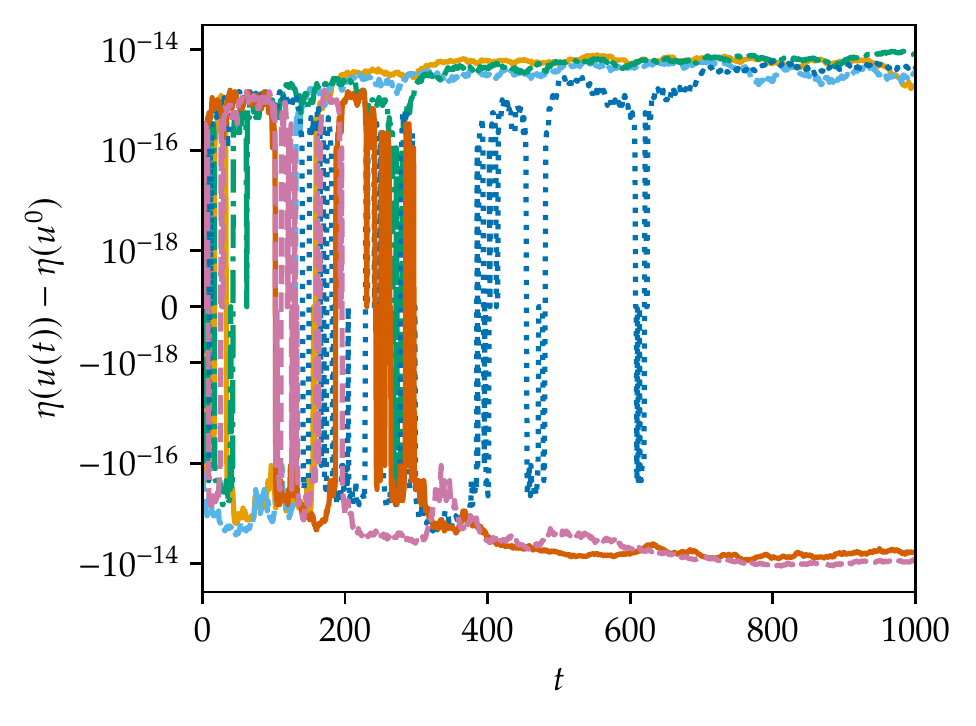}
    \caption{IDT methods.}
    \label{fig:pendulum__energy_relaxed_olddt}
  \end{subfigure}%
  \caption{Evolution of the non-quadratic energy \eqref{eq:pendulum_energy} of
           numerical solutions for the nonlinear pendulum \eqref{eq:pendulum}.}
  \label{fig:pendulum__energy}
\end{figure}

The energy of numerical solutions of \eqref{eq:pendulum} with $\dt = 0.9$ is
shown in Figure~\ref{fig:pendulum__energy}. As can be seen there, the energy
deviates significantly for all unmodified schemes while it is conserved to
machine accuracy for the RRK and IDT methods, as expected.

Typical results for this problem are shown in Figure~\ref{fig:pendulum__phase}.
Explicit methods tend to either create energy and drift away from the origin
such as SSPRK(3,3) or to dissipate energy and drift towards the origin such
as RK(4,4). In contrast, the corresponding relaxation schemes stay on the solution
manifold with constant energy and show qualitatively correct long time behavior.

\begin{figure}
\centering
  \begin{subfigure}[b]{0.8\textwidth}
    \centering
    \includegraphics[width=\textwidth]{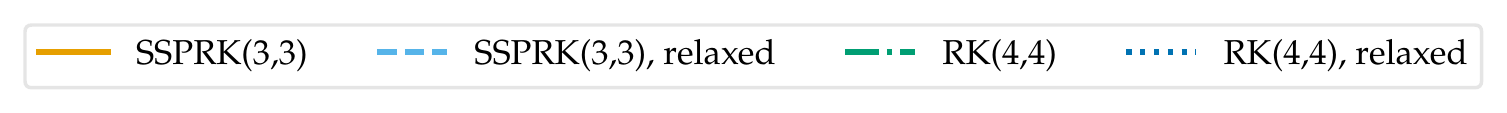}
  \end{subfigure}%
  \\
  \begin{subfigure}[b]{0.5\textwidth}
    \centering
    \includegraphics[width=\textwidth]{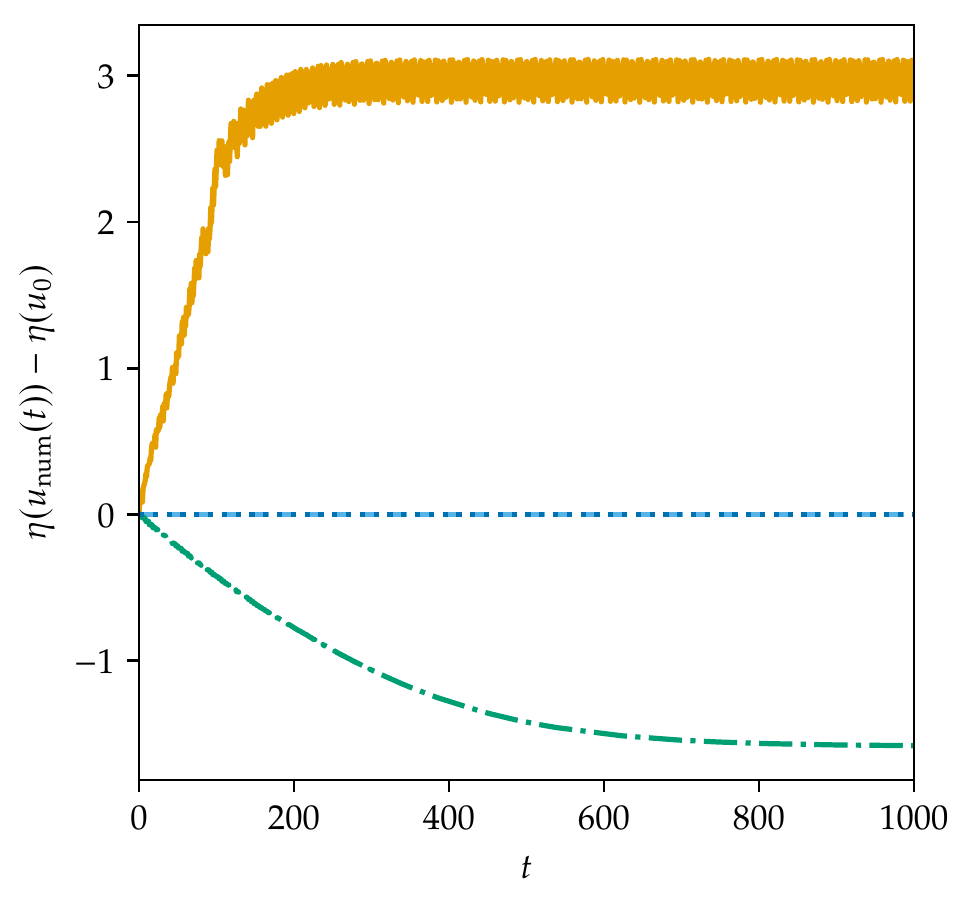}
    \caption{Energy over time.}
    \label{fig:pendulum__phase_energy}
  \end{subfigure}%
  \begin{subfigure}[b]{0.5\textwidth}
    \centering
    \includegraphics[width=\textwidth]{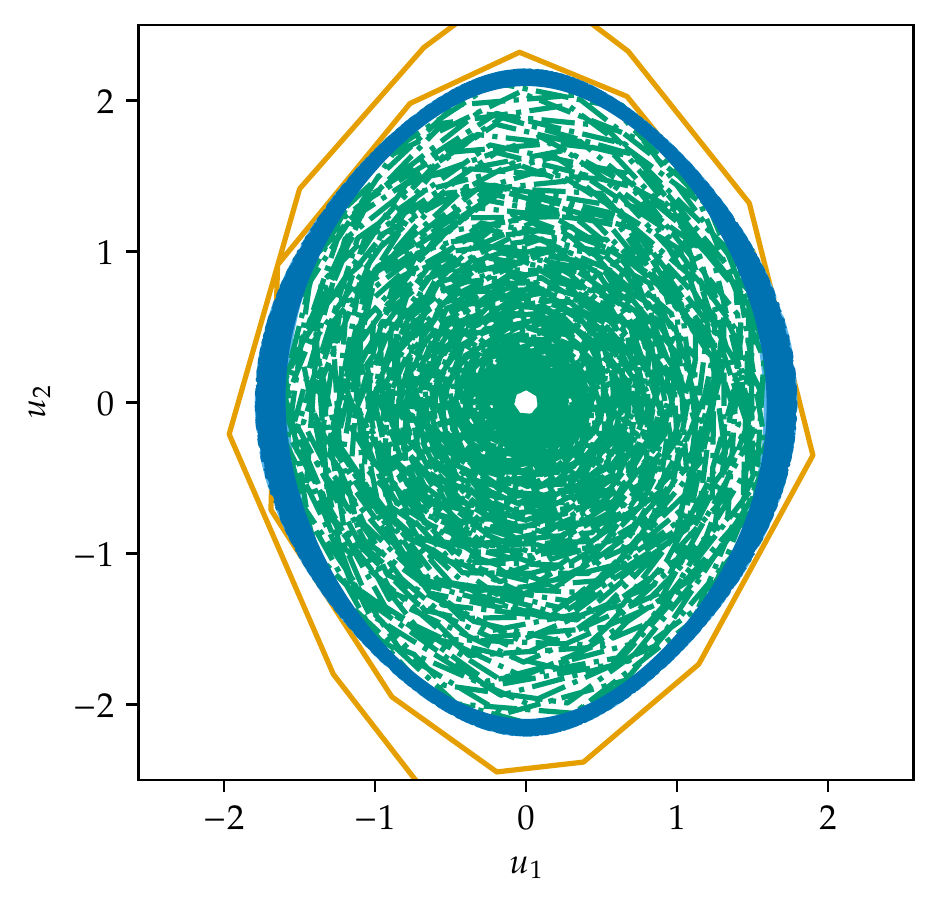}
    \caption{Phase space.}
    \label{fig:pendulum__phase_plot}
  \end{subfigure}%
  \caption{Numerical solutions for the nonlinear pendulum \eqref{eq:pendulum}
           using the unmodified and relaxation versions of SSPRK(3,3) and RK(4,4)
           with $\dt = 0.9$.}
  \label{fig:pendulum__phase}
\end{figure}

\subsection{Other Equations}
\label{sec:other-eqs}

Other systems such as the Lotka--Volterra equations with convex Lyapunov function,
the harmonic oscillator with quartic entropy $\eta(u) = \norm{u}^4$, and Burgers'
equation with a logarithmic entropy have also been tested. The results are
qualitatively similar to those presented above and can be found in the accompanying
repository \cite{ranocha2019relaxationRepository}.

\section{The Compressible Euler and Navier--Stokes Equations}\label{subsec:conservation_laws}
In this section, we apply the relaxation time integration schemes to the compressible
Euler and Navier--Stokes equations, which can be written as
\begin{equation}\label{eq:compressible_ns}
\begin{aligned}
&\frac{\partial\Q}{\partial t}+\sum\limits_{m=1}^{3}\frac{\partial \Fxm}{\partial \xm} =\sum\limits_{m=1}^{3}\frac{\partial\Fxmv}{\partial\xm}, &&
\forall \left(\xone,\xtwo,\xthree\right)\in\Omega,\quad t\ge 0,\\
&\Q\left(\xone,\xtwo,\xthree,t\right)=\GB\left(\xone,\xtwo,\xthree,t\right),&& \forall \left(\xone,\xtwo,\xthree\right)\in\Gamma,\quad t\ge 0,\\
&\Q\left(\xone,\xtwo,\xthree,0\right)=\Gzero\left(\xone,\xtwo,\xthree,0\right), &&
\forall \left(\xone,\xtwo,\xthree\right)\in\Omega.
\end{aligned}
\end{equation}
The vectors $\Q$, $\Fxm$, and $\Fxmv$ respectively denote the conserved variables,
the inviscid ($I$) fluxes, and the viscous ($V$) fluxes. The boundary data,
$\GB$, and the initial condition, $\Gzero$, are assumed to be in $L^{2}(\Omega)$,
with the further assumption that $\GB$ will be set to coincide with linear
well-posed boundary conditions and such that entropy conservation or stability is achieved.
The compressible Euler equations can be obtained from \eqref{eq:compressible_ns}
by setting $\Fxmv=0$.

It is well known that the compressible Navier--Stokes equations \eqref{eq:compressible_ns} possess
a convex extension that, when integrated over the physical domain $\Omega$,
only depends on the boundary data on $\Gamma$. Such an extension yields the entropy function
\begin{equation}\label{eq:entropy_function}
S=-\rho s,
\end{equation}
where $\rho$ and $s$ are the density and the thermodynamic entropy, respectively.
The entropy function, $S$, is convex with $S'' > 0$ if the thermodynamic variables
are positive and is a useful tool for proving stability in the $L^{2}$
norm \cite{dafermos_book_2010,tadmor_2006_ns}.

Following the analysis described in \cite{carpenter_ssdc_2014,parsani_entropy_stability_solid_wall_2015,carpenter_entropy_stability_ssdc_2016,fernandez_pref_euler_entropy_stability_2019},
we multiply multiply the PDE \eqref{eq:compressible_ns} by the (local) entropy
variables $\W = \partial S / \partial \Q$ and arrive at the the
integral form of the (scalar) entropy equation
\begin{equation}
 \label{eq:continuous_entropy_estimate_discont}
  \begin{split}
  \frac{\mr{d}}{\mr{d}t}\int_{\Omega}S\mr{d}\Omega = \frac{\mr{d}}{\mr{d}t}\eta
  &\le \sum\limits_{m=1}^{3}\int_{\Gamma}\left(\W\Tr \Fxmv - \fnc{F}_{\xm} \right)\nxm\mr{d}\Gamma - DT,
\end{split}
\end{equation}
where $\nxm$ is the $m$-th component of the outward facing unit normal to $\Gamma$ and
\begin{equation}\label{eq:DT}
DT = \sum_{m,j=1}^3 \int_{\Omega} \left(\frac{\partial \W}{\partial x_m}\right)^{\top} \Cij{m}{j} \,
  \frac{\partial \W}{\partial x_j} \mr{d}\Omega.
\end{equation}

We remark that viscous dissipation always introduces a negative rate of change in entropy,
since the $-DT$ term in \eqref{eq:continuous_entropy_estimate_discont} is negative semi-definite.
An increase in entropy within the domain can only result from
data that convects or diffuses through the boundary $\Gamma$.
For smooth flows, we note that the inequality sign in \eqref{eq:continuous_entropy_estimate_discont} becomes an equality.
Finally, we highlight that the integral form of the entropy equation for the compressible Euler equations can be obtained
from \eqref{eq:continuous_entropy_estimate_discont} by removing all the viscous terms.

Since our focus in the present work is on new time discretizations, we give
only a brief explanation of the spatial discretization.
We partition the physical domain $\Omega$ with boundary $\Gamma$
into non-overlapping hexahedral elements and we discretize the spatial terms
using a multi-dimensional SBP
simultaneous-approximation-terms (SBP-SAT) operator as described in
\cite{carpenter_ssdc_2014,parsani_entropy_stability_solid_wall_2015,fernandez_pref_euler_entropy_stability_2019,fernandez_pref_ns_entropy_stability_2019},
where the interested reader can find the details of the spatial discretization.

Using an SBP operator and its equivalent telescoping form
and following closely the entropy stability analysis presented in \cite{carpenter_ssdc_2014,parsani_entropy_stability_solid_wall_2015,carpenter_entropy_stability_ssdc_2016}, the total entropy of the spatial
discretization satisfies
\begin{equation}\label{eq:estimate-no-slip-bc-2}
  \od{}{t} \mathbf{1}^{\top} \widehat{\Pmatvol} \, \bm{S} = \od{}{t} \eta
	= \mathbf{BT} - \mathbf{DT} + \mathbf{\Upsilon}.
\end{equation}
This equation mimics at the semi-discrete level each term in \eqref{eq:continuous_entropy_estimate_discont}.
Here $\mathbf{BT}$ is the discrete
boundary term (i.e., the discrete version of the first integral term on the right-hand side of \eqref{eq:continuous_entropy_estimate_discont}),
$\mathbf{DT}$ is
the discrete dissipation term (i.e., the discrete version of the second term on the right-hand side of \eqref{eq:continuous_entropy_estimate_discont})
and $\mathbf{\Upsilon}$ enforces interface coupling and boundary conditions \cite{carpenter_ssdc_2014,parsani_entropy_stability_solid_wall_2015,carpenter_entropy_stability_ssdc_2016}. For completeness, we note that the matrix
$\widehat{\Pmatvol}$ may be thought of as the mass matrix in the context of the discontinuous Galerkin
finite element method.

In the next part of this section, six test cases will be considered. The first one is the propagation
of an isentropic vortex for the compressible Euler equations. This test case is used to i) perform
a convergence study of the combined space and time discretizations for the compressible Euler equations and
ii) verify the entropy conservative properties of the full discretization. The second test case
is the propagation of a viscous shock and is used to assess the accuracy properties
of the complete entropy stable discretization for the compressible Navier--Stokes
equations. The third and fourth test cases are the Sod's shock tube and the
sine-shock interaction of Titarev and Toro \cite{titarev_2014} which is
the extension of the Shu--Osher problem with much more severe oscillations. These
two test cases are used to show the robustness of the fully-discrete entropy stable
algorithm for non-smooth solutions \cite{carpenter_ssdc_2014,Parsani2016}.
The fifth test case is the
laminar flow in a lid-driven cavity where a non-zero heat entropy flux is imposed on one of the vertical
faces of the cavity. This test case is used to show the capabilities of the full
discretization to capture correctly the time evolution
of the entropy when, for instance, non-homogeneous boundary conditions are imposed.
Finally, the supersonic turbulent flow past a rod of square section \cite{parsani_entropy_stability_solid_wall_2015}
is used to demonstrate algorithmic robustness for the compressible Navier--Stokes
equations.

The error is computed using the following norms:
\begin{equation}
\begin{aligned}
  \text{Discrete }L^{1}&:& \|\bm{q}\|_{L^{1}}&=\Omega_{c}^{-1}\sum\limits_{j=1}^{N_{el}}\onek\Tr\Mk\matJk\textrm{abs}\left(\qk\right),\\
  \text{Discrete }L^{2}&:& \|\bm{q}\|_{L^{2}}^2&=\Omega_{c}^{-1}\sum\limits_{j=1}^{N_{el}}\qk\Tr\Mk\matJk\qk,\\
  \text{Discrete }L^{\infty}&:& \|\bm{q}\|_{L^{\infty}}&=\max\limits_{j=1\dots N_{el}}\textrm{abs}\left(\qk\right).
\end{aligned}
\end{equation}
Here $\matJk$ is the metric Jacobian of the curvilinear transformation from physical
space to computational space of the $j$-th hexahedral element, $N_{el}$ is the
total number of hexahedral elements in the mesh. Furthermore, $\Omega_{c}$ indicates the volume of $\Omega$ computed as
$\Omega_{c}\equiv\sum\limits_{\kappa=1}^{K}\onek_{\kappa}\Tr\M^{\kappa}\matJk\onek_{\kappa}$,
where $\onek_{\kappa}$ is a vector of ones of the size of the number of nodes on the
$\kappa$-th element.

The unstructured grid solver used herein has been developed at the Extreme
Computing Research Center (ECRC) at KAUST on top of the Portable and Extensible
Toolkit for Scientific computing (PETSc)~\cite{petsc-user-ref}, its mesh
topology abstraction (DMPLEX)~\cite{KnepleyKarpeev09} and scalable ordinary
differential equation (ODE)/differential algebraic equations (DAE) solver
library~\cite{abhyankar2018petsc}. The parameter $\gamma_{n}$ of the relaxation
Runge--Kutta schemes is computed
from Equation \eqref{eq:r} to machine precision using the bisection method
which, for efficiency, is implemented directly in the unstructured grid solver.

\subsection{Propagation of an Isentropic Vortex in Three Dimensions}\label{subsubsec:iv}

\begin{figure}
\begin{center}
   \includegraphics[width=0.5\textwidth]{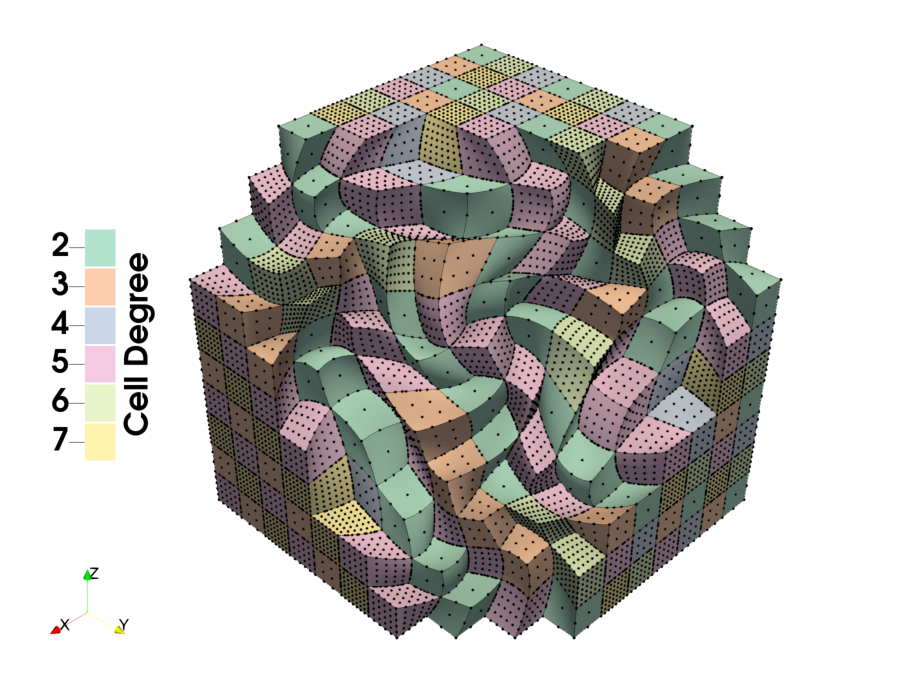}
   \caption{Isentropic vortex: mesh cut and polynomial degree distribution with non-conforming interfaces; $p=2$ to $p=7$.}
   \label{fig:iv_mesh_cut_p_ref}
  \end{center}
\end{figure}

In this section, we investigate the accuracy and the entropy conservation property
of the full discretization obtained by combining SBP-SAT
entropy conservative operators and relaxation time integration
schemes. To do so, we simulate the
propagation of an isentropic vortex by solving the three-dimensional compressible
Euler equations. The analytical solution of this problem is
\begin{equation} \label{vortex-solution}
\begin{split}
& \fnc{G} = 1
-\left\{
\left[
\left(\xone-x_{1,0}\right)
-U_{\infty}\cos\left(\alpha\right)t
\right]^{2}
+
\left[
\left(\xtwo-x_{2,0}\right)
-U_{\infty}\sin\left(\alpha\right)t
\right]^{2}
\right\},\\
&\rho = T^{\frac{1}{\gamma-1}},
 \quad
 T = \left[1-\epsilon_{\nu}^{2}M_{\infty}^{2}\frac{\gamma-1}{8\pi^{2}}\exp\left(\fnc{G}\right)\right],\\
&\Uone = U_{\infty}\cos(\alpha)-\epsilon_{\nu}
\frac{\left(\xtwo-x_{2,0}\right)-U_{\infty}\sin\left(\alpha\right)t}{2\pi}
\exp\left(\frac{\fnc{G}}{2}\right),\\
&\Utwo = U_{\infty}\sin(\alpha)-\epsilon_{\nu}
\frac{\left(\xone-x_{1,0}\right)-U_{\infty}\cos\left(\alpha\right)t}{2\pi}
\exp\left(\frac{\fnc{G}}{2}\right),
\quad \Uthree = 0,
\end{split}
\end{equation}
where $U_{\infty}$ is the modulus of the free-stream velocity, $M_{\infty}$
is the free-stream Mach number, $c_{\infty}$ is the free-stream speed of sound,
and $\left(x_{1,0},x_{2,0},x_{3,0}\right)$ is the vortex center.
The following values are used: $U_{\infty}=M_{\infty} c_{\infty}$, $\epsilon_{\nu}=5$, $M_{\infty}=0.5$,
$\gamma=1.4$, $\alpha=45^{\degree}$, and $\left(x_{1,0},x_{2,0},x_{3,0}\right)=\left(0,0,0\right)$.
The computational domain is
\begin{equation*}
	\xone\in[-5,5], \qquad\xtwo\in[-5,5],\qquad\xthree\in[-5,5],\qquad  t\in[0,10].
\end{equation*}
The initial condition is given by \eqref{vortex-solution} with $t=0$.
Periodic boundary conditions are used on all six faces of the computational domain.
First, we run a convergence study for the complete entropy-stable discretization
by simultaneously refining the grid spacing and the time step and keeping the
ratio $U_{\infty}\Delta t/\Delta x$ constant and equal to $0.05$.
The errors and convergence rates in the $L^1$, $L^2$ and $L^{\infty}$ norms for fourth-, fifth-,
sixth-order accurate algorithms are reported in Table~\ref{tab:conv_iv_p3to5_relax}.
We observe that the computed order of convergence in both $L^1$ and $L^2$ norms matches
the design order of the scheme.

\begin{table}
\centering
  \caption{Convergence study for the isentropic vortex using entropy conservative
  SBP-SAT schemes with different solution polynomial degrees $p$ and relaxation Runge--Kutta
  methods ($U_{\infty}\Delta t/\Delta x = 0.05$, error in the density).}
  \label{tab:conv_iv_p3to5_relax}
  \begin{tabular*}{\linewidth}{@{\extracolsep{\fill}}*8c@{}}
    \toprule
    $p$ & RK Method
    & $L^1$ Error & $L^1$ Rate
    & $L^2$ Error & $L^2$ Rate
    & $L^\infty$ Error & $L^\infty$ Rate
    \\
    \midrule
    3 & RK(4,4) &
      2.66E-03  & ---    & 1.36E-04  & ---    & 2.46E-02     & ---   \\ &&
      2.15E-04  &  3.63  & 1.20E-05  &  3.50  & 3.18E-03     &  2.95 \\ &&
      1.29E-05  &  4.06  & 8.50E-07  &  3.82  & 3.49E-04     &  3.19 \\ &&
      6.60E-07  &  4.29  & 5.21E-08  &  4.03  & 2.66E-05     &  3.71 \\ &&
      3.84E-08  &  4.10  & 2.82E-09  &  4.21  & 2.04E-06     &  3.70 \\
    \midrule
    4 & BSRK(8,5) &
      3.34E-04  & ---   & 4.57E-05  & ---   & 8.50E-03     & ---   \\ &&
      3.08E-05  & 4.76  & 2.05E-06  & 4.48  & 9.06E-04     & 3.23  \\ &&
      7.33E-07  & 5.39  & 5.62E-08  & 5.19  & 5.80E-05     & 3.97  \\ &&
      2.05E-08  & 5.16  & 1.71E-09  & 5.04  & 1.33E-06     & 5.45  \\ &&
      5.70E-10  & 5.17  & 4.76E-11  & 5.17  & 3.38E-08     & 5.30  \\
    \midrule
    5 & VRK(9,6) &
      2.23E-04  & ---   & 1.31E-05  & ---   & 3.39E-03     & --    \\ &&
      3.55E-06  & 5.98  & 2.31E-07  & 5.82  & 8.69E-05     & 5.29  \\ &&
      6.74E-08  & 5.72  & 4.87E-09  & 5.57  & 3.25E-06     & 4.74  \\ &&
      1.10E-09  & 5.93  & 6.81E-11  & 6.16  & 7.48E-08     & 5.44  \\ &&
      1.70E-11  & 6.02  & 9.57E-13  & 6.15  & 1.64E-09     & 5.51  \\
    \bottomrule
  \end{tabular*}
\end{table}

Next, we validate the full entropy-conservative property by simulating
the propagation of the isentropic vortex using a grid with ten hexahedra in each
coordinate direction and
non-conforming interfaces (see Figure \ref{fig:iv_mesh_cut_p_ref}). The grid is generated by setting
the solution polynomial degree in each element to a random integer chosen uniformly from the set $\{2,3,4,5\}$
\cite{fernandez_pref_euler_entropy_stability_2019}.\footnote{This corresponds to
SBP-SAT operators which are formally third to sixth order accurate.}
All the dissipation terms used for the interface coupling \cite{parsani_entropy_stable_interfaces_2015,fernandez_pref_euler_entropy_stability_2019} are
turned off, including upwind and interior-penalty SATs.
To highlight that the space and time discretizations and their coupling are
truly entropy conservative, we compute in quadruple precision.

In addition to the Runge--Kutta methods mentioned at the beginning of
Section~\ref{sec:numerical}, we use the following methods, that also have weights
$b_i\geq 0$. Again, the value of $\Delta t$ is fixed in each test, and embedded error
estimators are not used.
\begin{itemize}
  \item
  LSCKRK(5,4): Five stage, fourth order method of \cite{CK1994}.

  \item
  BSRK(3,3): Three stage, third order method of \cite{bogacki1989third}.

  \item
  BSRK(7,5): Seven stage, fifth order method of \cite{bogacki1996efficient}.

  \item
  VRK(10,7): Ten stage, seventh order method of the family developed
  in \cite{verner1978explicit}\footnote{The coefficients are taken from \url{http://people.math.sfu.ca/~jverner/RKV76.IIa.Robust.000027015646.081206.CoeffsOnlyFLOAT}
  at 2019-05-02.}.
\end{itemize}

\begin{figure}
  \begin{center}
  \includegraphics[width=0.65\textwidth]{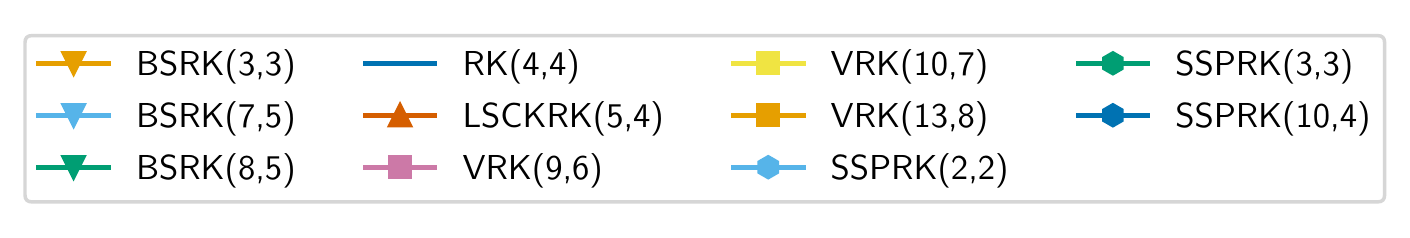}\\
  \begin{subfigure}[b]{.45\textwidth}
    \includegraphics[width=\textwidth]{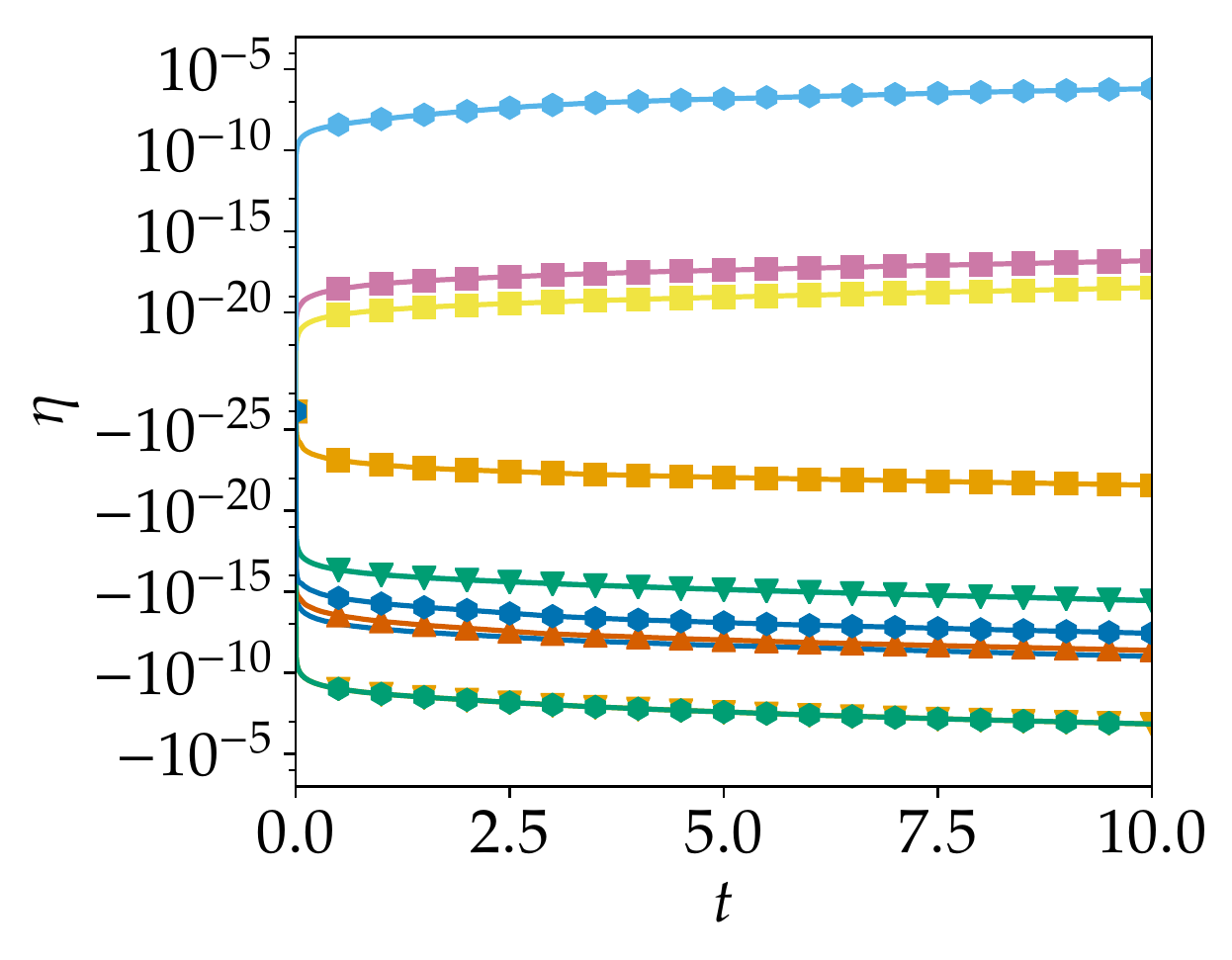}
    \caption{Without relaxation.}
    \label{fig:iv_no_relaxation_non-conf_p_ref}
   \end{subfigure}
  \begin{subfigure}[b]{.435\textwidth}
   \includegraphics[width=\textwidth]{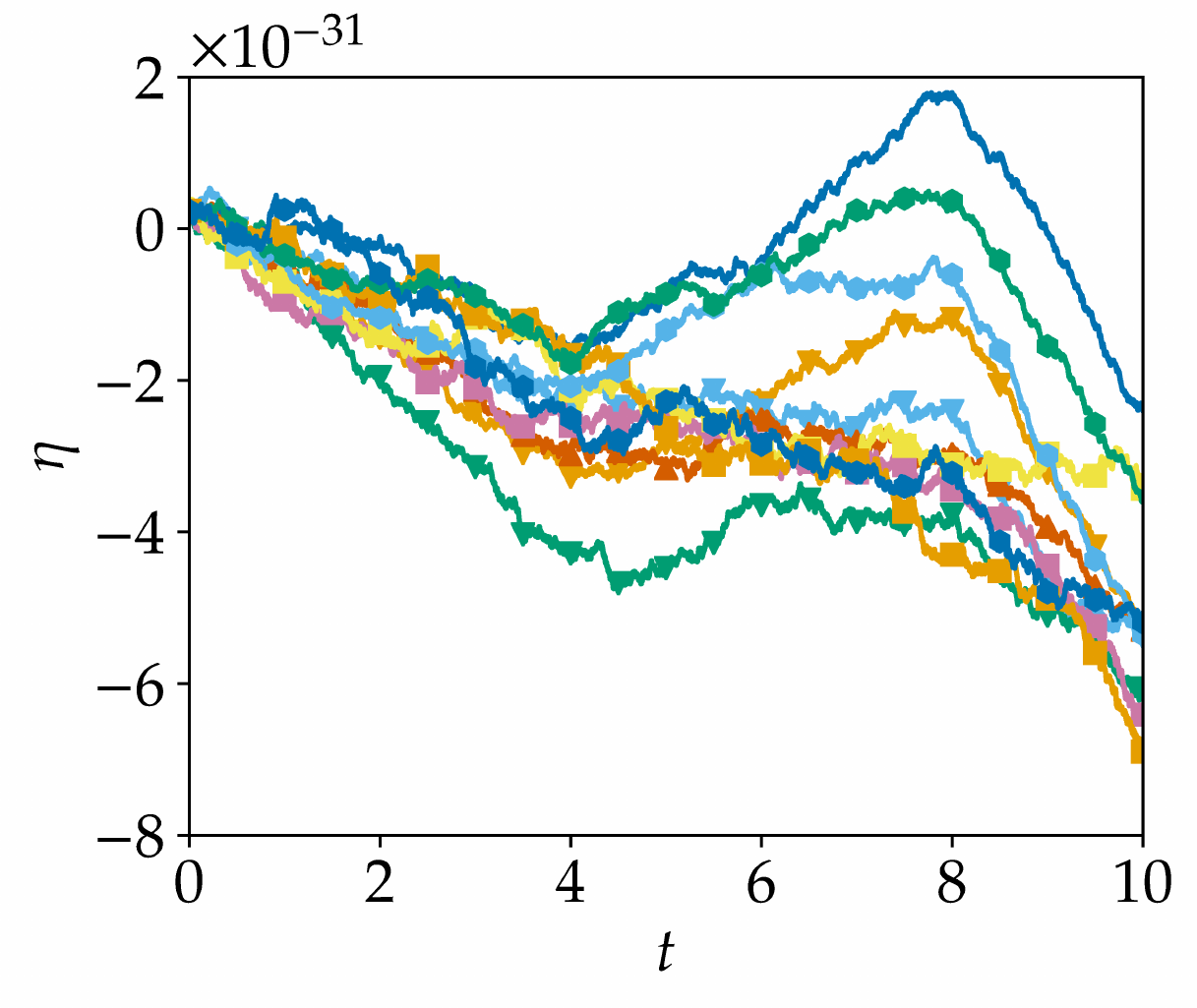}
   \caption{With relaxation.}
   \label{fig:iv_relaxation_non-conf_p_ref}
   \end{subfigure}
    \caption{Isentropic vortex: time evolution of total entropy, using
             the spatial discretization shown in Figure \ref{fig:iv_mesh_cut_p_ref}.}
    \label{fig:iv_non-conf_p_ref}
  \end{center}
\end{figure}

We show the entropy variation with and without relaxation in
Figure \ref{fig:iv_non-conf_p_ref}. The entropy is conserved up to machine (quadruple)
precision using relaxation, whereas, without relaxation, all solutions show
significant changes in total entropy.

\subsection{Three-Dimensional Propagation of a Viscous Shock}\label{subsubsec:vs}
Next we study the propagation of a viscous shock using the
compressible Navier--Stokes equations.  We assume a planar shock propagating along the $\xone$
coordinate direction with a Prandtl number of $Pr=3/4$.
The exact solution of this problem is known;
the momentum $\fnc{V}(x_1)$ satisfies the ODE
\begin{equation}
\begin{split}
&\alpha\fnc{V}\frac{\partial\fnc{V}}{\partial\xone}-(\fnc{V}-1)(\fnc{V}-\fnc{V}_{f})=0, \qquad -\infty\leq\xone\leq+\infty
\end{split}
\end{equation}
whose solution can be written implicitly as\footnote{The integration constant is taken
equal to zero because the center of the viscous shock is assumed to be at $\xone = 0$.}
\begin{equation}\label{eq:implicit_sol_vs}
\xone-\frac{1}{2}\alpha\left(\log\left|(\fnc{V}(x_1)-1)(\fnc{V}(x_1)-\fnc{V}_{f})\right|+\frac{1+\fnc{V}_{f}}{1-\fnc{V}_{f}}\log\left|\frac{\fnc{V}(x_1)-1}{\fnc{V}(x_1)-\fnc{V}_{f}}\right|\right) = 0,
\end{equation}
where
\begin{equation}
\fnc{V}_{f}\equiv\frac{\fnc{U}_{L}}{\fnc{U}_{R}},\qquad
\alpha\equiv\frac{2\gamma}{\gamma + 1}\frac{\,\mu}{Pr\dot{\fnc{M}}}.
\end{equation}
Here $\fnc{U}_{L/R}$ are known velocities to the left and right of the shock at
$-\infty$ and $+\infty$, respectively, $\dot{\fnc{M}}$ is the constant mass
flow across the shock, $Pr$ is the Prandtl number, and $\mu$ is the dynamic
viscosity.
The mass and total enthalpy are constant across the shock. Moreover,
the momentum and energy equations become redundant.

For our tests, $\fnc{V}$ is computed from Equation \eqref{eq:implicit_sol_vs}
to machine precision using bisection.
The moving shock solution is obtained by applying a uniform translation to the above solution.
The shock is located at the center of the domain at $t=0$ and the following
values are used: $M_{\infty}=2.5$, $Re_{\infty}=10$, and $\gamma=1.4$.
The domain is given by
\[
	\xone\in[-0.5,0.5], \qquad\xtwo\in[-0.5,0.5],\qquad\xthree\in[-0.5,0.5],\qquad t\in[0,0.5].
\]
The boundary conditions are prescribed by penalizing the numerical solution
against the exact solution. The analytical solution is also used to furnish data
for the initial conditions.

We run a convergence study for the complete entropy stable discretization
by simultaneously refining the grid spacing and the time step and keeping the
ratio $U_{\infty}\Delta t/\Delta x^2$ constant and equal to $0.05$.
The errors and convergence rates in the $L^1$, $L^2$ and $L^{\infty}$ norms for fourth-, fifth-,
sixth-order accurate algorithms are reported in Table~\ref{tab:conv_vs_p3to5_relax}.
As for the compressible Euler equations, we observe that the
order of convergence in both $L^1$ and $L^2$ norms is the expected one.

\begin{table}
\centering
  \caption{Convergence study for the viscous shock using entropy stable
  SBP-SAT schemes with different solution polynomial degrees $p$ and relaxation Runge--Kutta
  methods ($U_{\infty}\Delta t/\Delta x^2 = 0.05$, error in the density).}
  \label{tab:conv_vs_p3to5_relax}
  \begin{tabular*}{\linewidth}{@{\extracolsep{\fill}}*8c@{}}
    \toprule
    $p$ & RK Method
    & $L^1$ Error & $L^1$ Rate
    & $L^2$ Error & $L^2$ Rate
    & $L^\infty$ Error & $L^\infty$ Rate
    \\
    \midrule
    3 & RK(4,4) &
      2.59E-02  & ---   & 3.78E-02  & ---   & 1.11E-01     & ---  \\ &&
      1.88E-03  & 3.79  & 2.81E-03  & 3.75  & 9.77E-03     & 3.51 \\ &&
      1.03E-04  & 4.19  & 1.99E-04  & 3.82  & 9.89E-04     & 3.30 \\ &&
      5.90E-06  & 4.13  & 9.97E-06  & 4.32  & 6.12E-05     & 4.02 \\ &&
      3.30E-07  & 4.16  & 5.47E-07  & 4.19  & 3.92E-06     & 3.97 \\
    \midrule
    4 & BSRK(8,5) &
      6.80E-03  & ---   & 9.01E-03  & ---   & 2.00E-02     & ---   \\ &&
      5.74E-04  & 3.57  & 9.11E-04  & 3.31  & 4.02E-03     & 2.32  \\ &&
      2.78E-05  & 4.37  & 5.25E-05  & 4.12  & 3.32E-04     & 3.60  \\ &&
      6.30E-07  & 5.46  & 1.33E-06  & 5.30  & 1.06E-05     & 4.97  \\ &&
      1.70E-08  & 5.21  & 3.30E-08  & 5.33  & 3.59E-07     & 4.88  \\
    \midrule
    5 & VRK(9,6) &
      3.67E-03  & ---   & 6.17E-03  & ---   & 2.53E-02     & ---   \\ &&
      1.61E-04  & 4.51  & 2.57E-04  & 4.59  & 1.24E-03     & 4.35  \\ &&
      1.34E-06  & 6.90  & 2.93E-06  & 6.45  & 2.07E-05     & 5.91  \\ &&
      1.62E-08  & 6.37  & 3.90E-08  & 6.23  & 3.94E-07     & 5.71  \\
    \bottomrule
  \end{tabular*}
\end{table}


\subsection{Sod's Shock Tube}\label{subsec:sod}
Sod's shock tube problem is a classical Riemann problem that evaluates the behavior of a numerical method
when a shock, expansion, and contact discontinuity are present.
Of particular interest is smearing in the shock and contact, or
oscillations at any of the discontinuities. The governing equations are the
time-dependent one-dimensional compressible Euler equations which are solved
in the domain given by
\begin{equation*}
  x_1 \in [0,1], \quad t \in [0,0.2].
\end{equation*}

The problem is initialized with
\begin{equation*}
 \begin{gathered}
  \rho = \left\{\begin{array}{lr}
                 1 & x_1<0.5, \\
                 1/8 & x_1\geq 0.5,
                \end{array}\right. \quad
  p = \left\{\begin{array}{lr}
                 1     & x_1<0.5, \\
                 1 /10 & x_1\geq 0.5,
                \end{array}\right. \quad
  \Uone = 0,
 \end{gathered}
\end{equation*}
where $\rho$ and $p$ are the density and pressure, respectively.
All simulations used a ratio of specific heats equals to $7/5$.
\begin{figure}
\centering
  \begin{subfigure}[b]{0.49\textwidth}
    \centering
    \includegraphics[width=\textwidth]{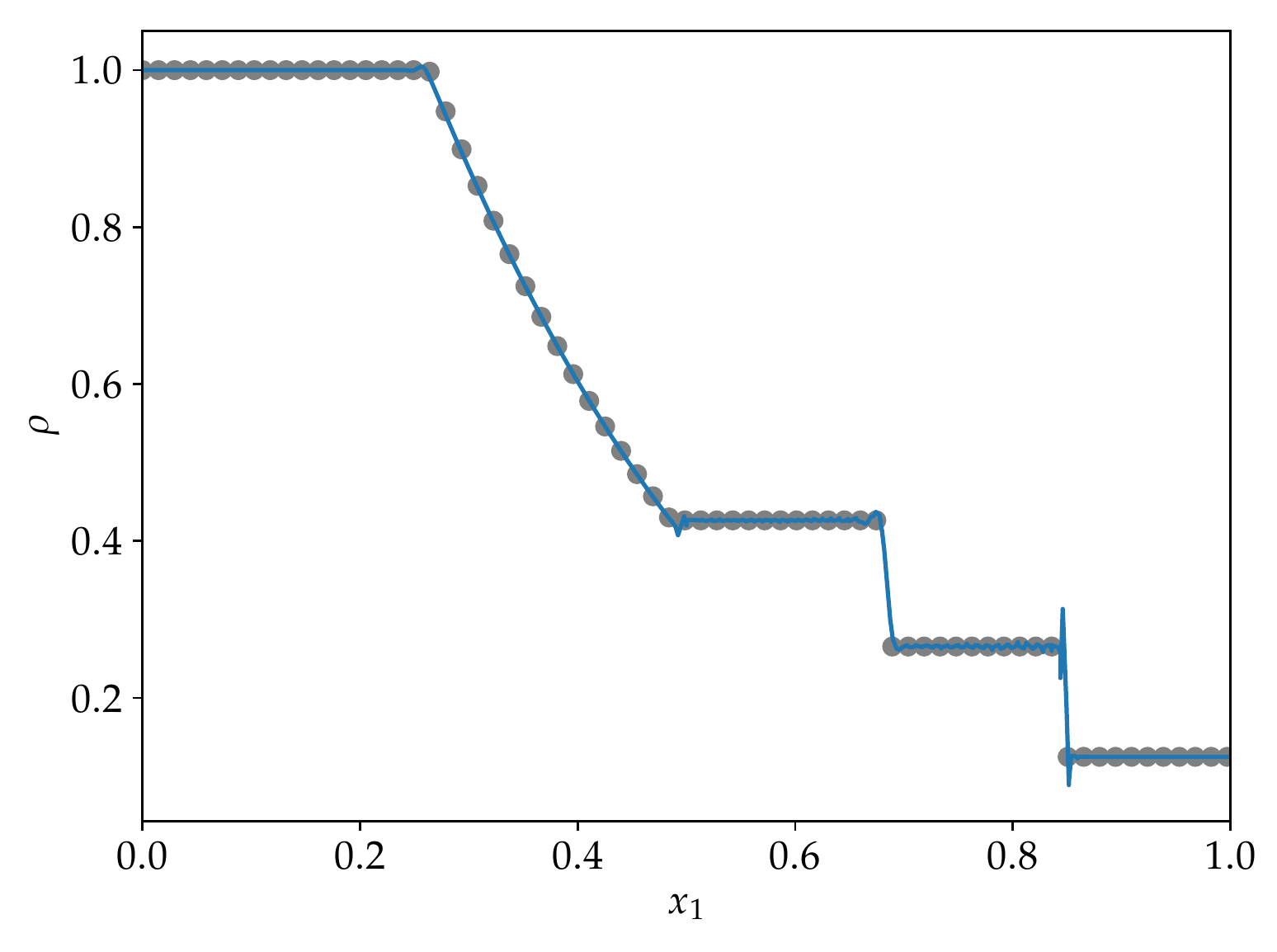}
    \caption{Without relaxation.}
    \label{fig:sod_without_relaxation}
  \end{subfigure}
  \begin{subfigure}[b]{0.49\textwidth}
    \centering
    \includegraphics[width=\textwidth]{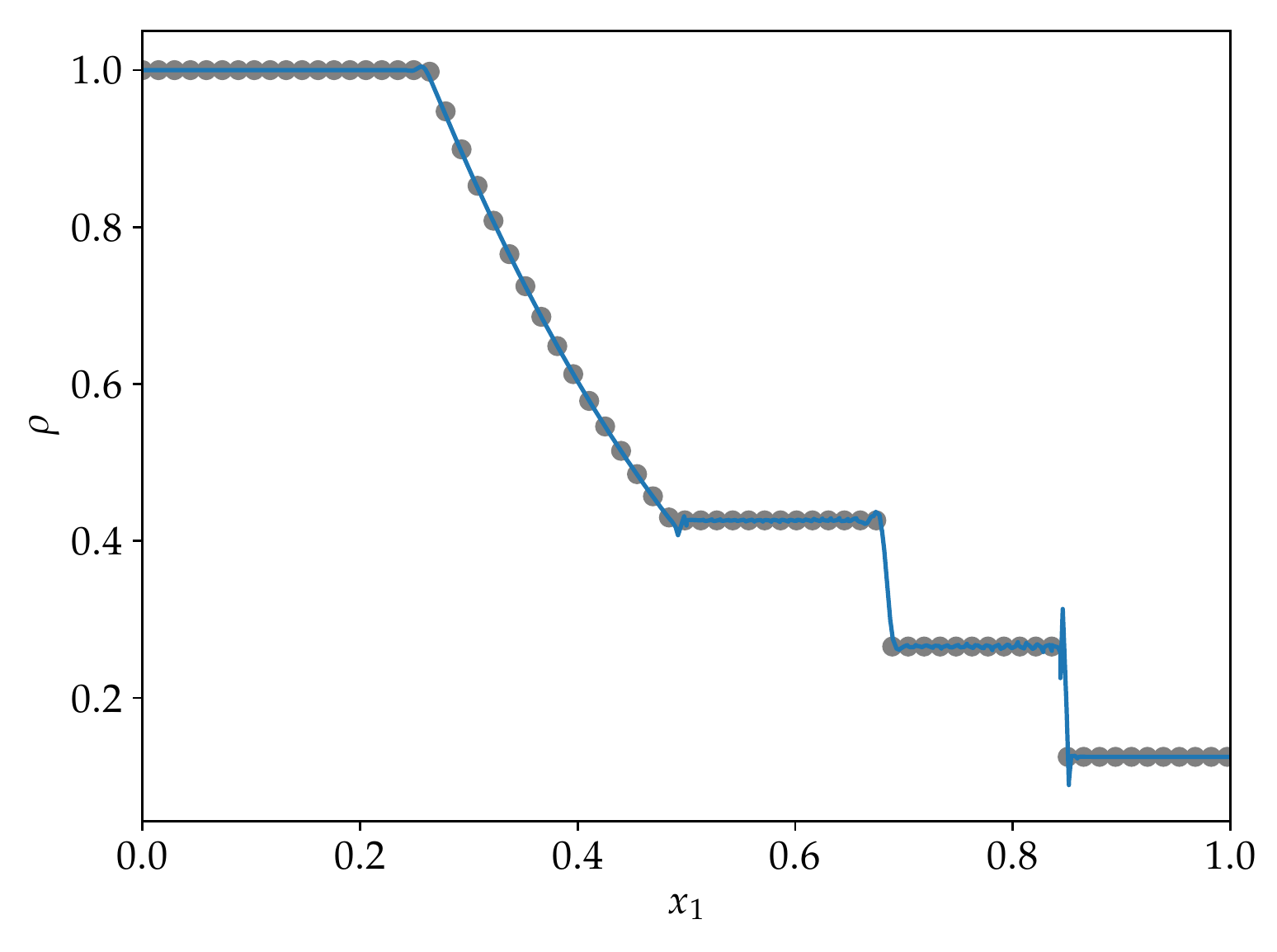}
    \caption{With relaxation.}
    \label{fig:sod_with_relaxation}
  \end{subfigure}%
  \caption{Density profile of Sod's shock tube problem (exact solution shown with circles).}
  \label{fig:sod}
\end{figure}

The entropy stable spatial discretization uses polynomials of degree $p=3$ and a grid
with $128$ elements. The problem is integrated in time using the classical
fourth-order accurate Runge--Kutta method RK(4,4).

Results of the density with and without relaxation are
visually indistinguishable, as shown in Figure~\ref{fig:sod}.
The relaxation approach does not increase the quality of the solution
and small overshoots near non-smooth parts of the numerical approximation
are visible. This behavior is expected for a spatial discretization which uses
high-order polynomials and no explicit shock capturing mechanism.

Nevertheless, the relaxation approach does also not decrease the quality
of the solution while guaranteeing the correct sign of the entropy evolution.
In particular, this guarantee does not result in excessive artificial
viscosity for shocks and the relaxation scheme does not smear the shock
solution for a high-order accurate SBP spatial scheme.

For this experiment, $\gamma$ deviates from unity by less than $5 \times 10^{-4}$,
as shown in Figure~\ref{fig:sod-gamma}. After a short initial period, the
value of $\gamma$ seems to oscillate following a regular pattern with
amplitude $\lesssim 10^{-4}$.

\begin{figure}
\centering
  \begin{subfigure}[b]{0.49\textwidth}
    \centering
    \includegraphics[width=\textwidth]{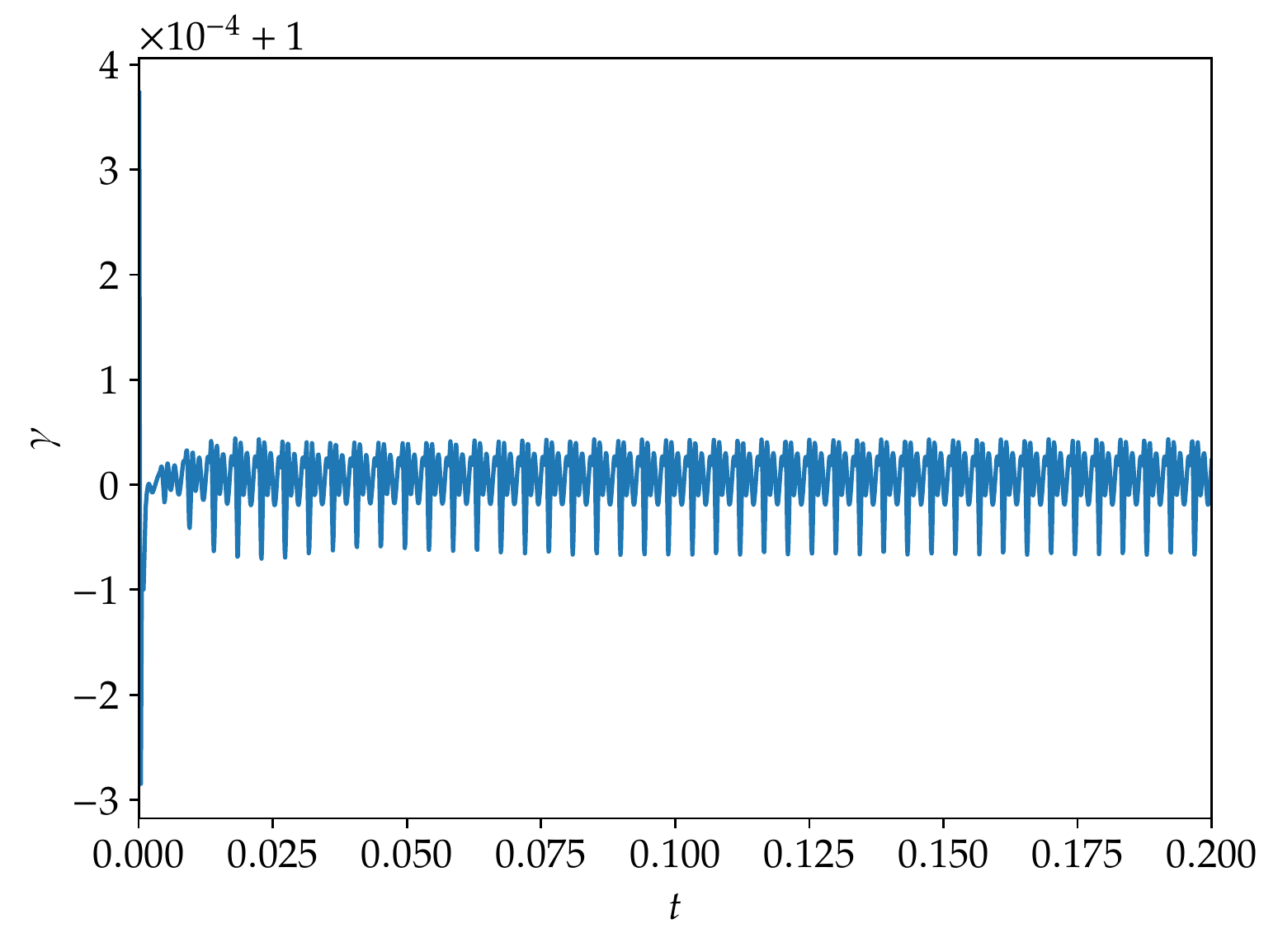}
    \caption{Sod's shock tube.}
    \label{fig:sod-gamma}
  \end{subfigure}
  \begin{subfigure}[b]{0.49\textwidth}
    \centering
    \includegraphics[width=\textwidth]{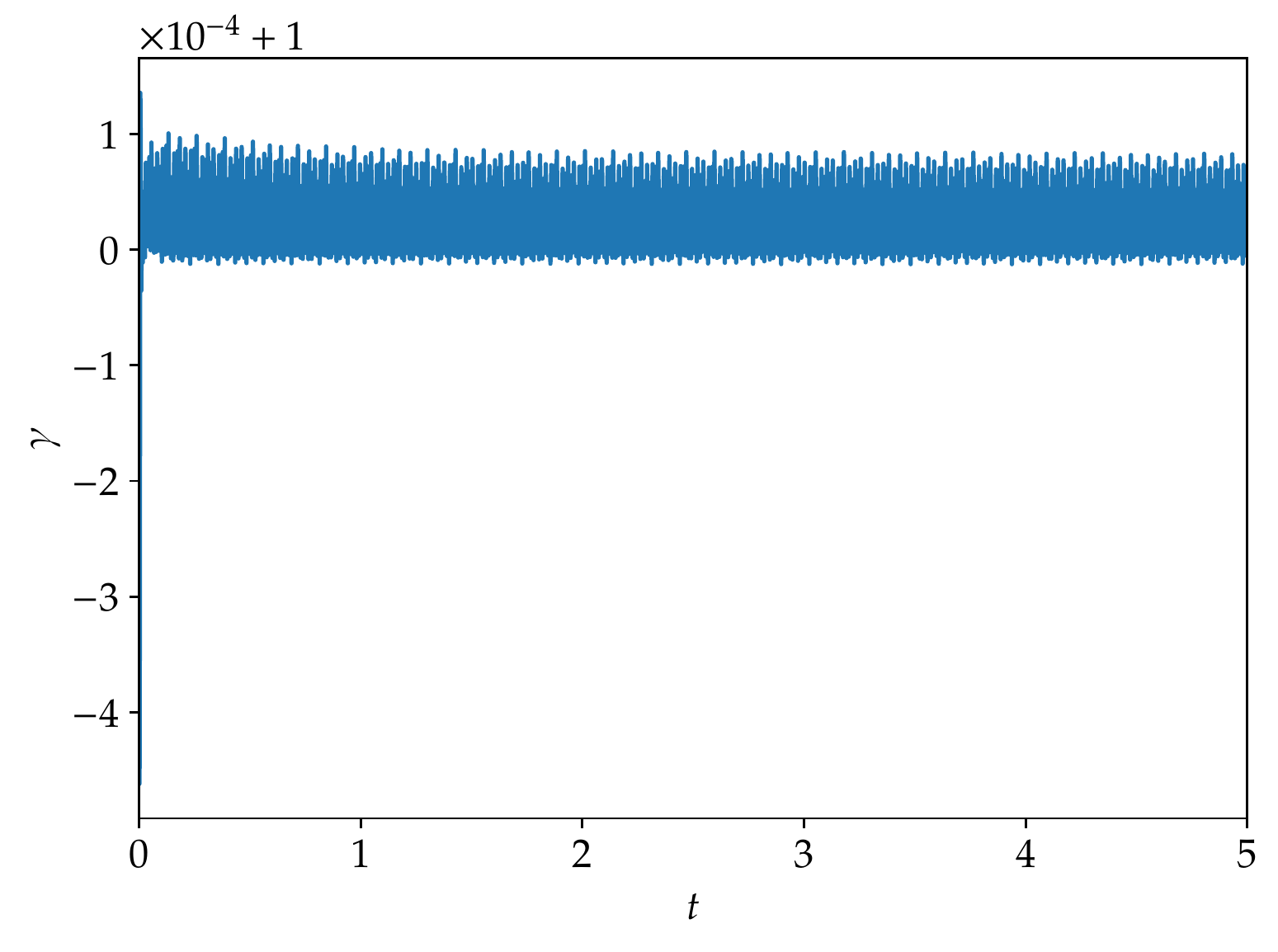}
    \caption{Sine-shock interaction.}
    \label{fig:sine-gamma}
  \end{subfigure}%
  \caption{Variation of the relaxation parameter, $\gamma$, for the shock problems.}
  \label{fig:sod-sine-gamma}
\end{figure}


\subsection{Sine-Shock Interaction}\label{subsec:sine}
The solution of this benchmark problem
contains both strong discontinuities and smooth structures and is well suited
for testing high-order shock-capturing schemes.
The governing equations are the
time-dependent one-dimensional compressible Euler equations which are solved
in the domain given by
\[
  x_1 \in [-5,+5], \quad t \in [0,5].
\]

The problem is initialized with \cite{titarev_2014}
\begin{equation*}
\left( \rho, \Uone, p \right) = \left\{
\begin{array}{lll}
\left(1.515695, \, 0523346, \, 1.805\right), & \quad \text{if} \quad -5   \le x < -4.5 \\
  \left(1 + 0.1 \sin (20\pi x), \, 0, \, 1 \right),       & \quad \text{if} \quad -4.5 \le x \le 5.
\end{array}
\right.
\end{equation*}
The exact solution to this problem is not available.
\begin{figure}
\centering
  \begin{subfigure}[b]{.49\textwidth}
    \centering
    \includegraphics[width=\textwidth]{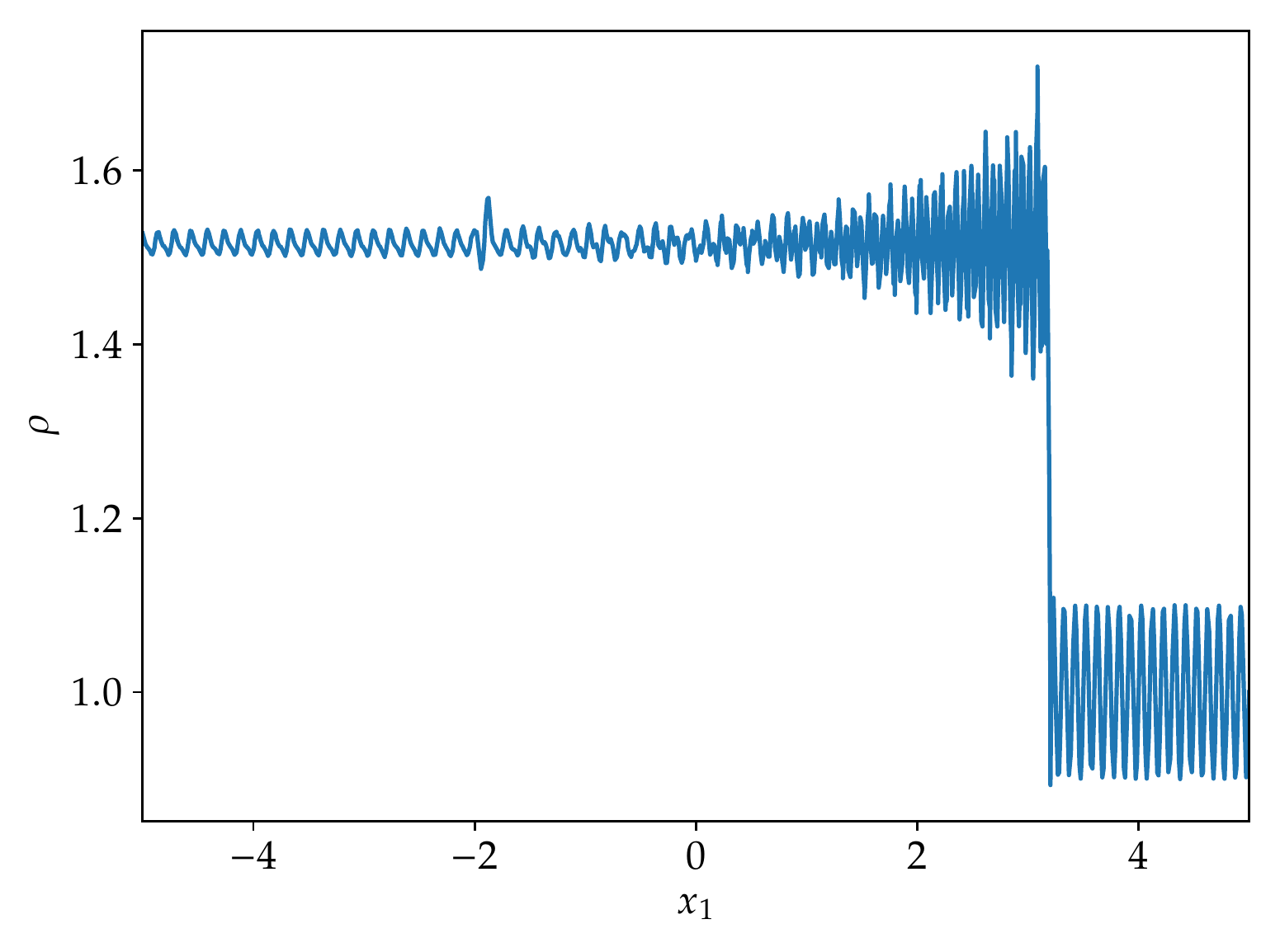}
    \caption{Without relaxation.}
    \label{fig:sine_without_relaxation}
  \end{subfigure}
  \begin{subfigure}[b]{.49\textwidth}
    \centering
    \includegraphics[width=\textwidth]{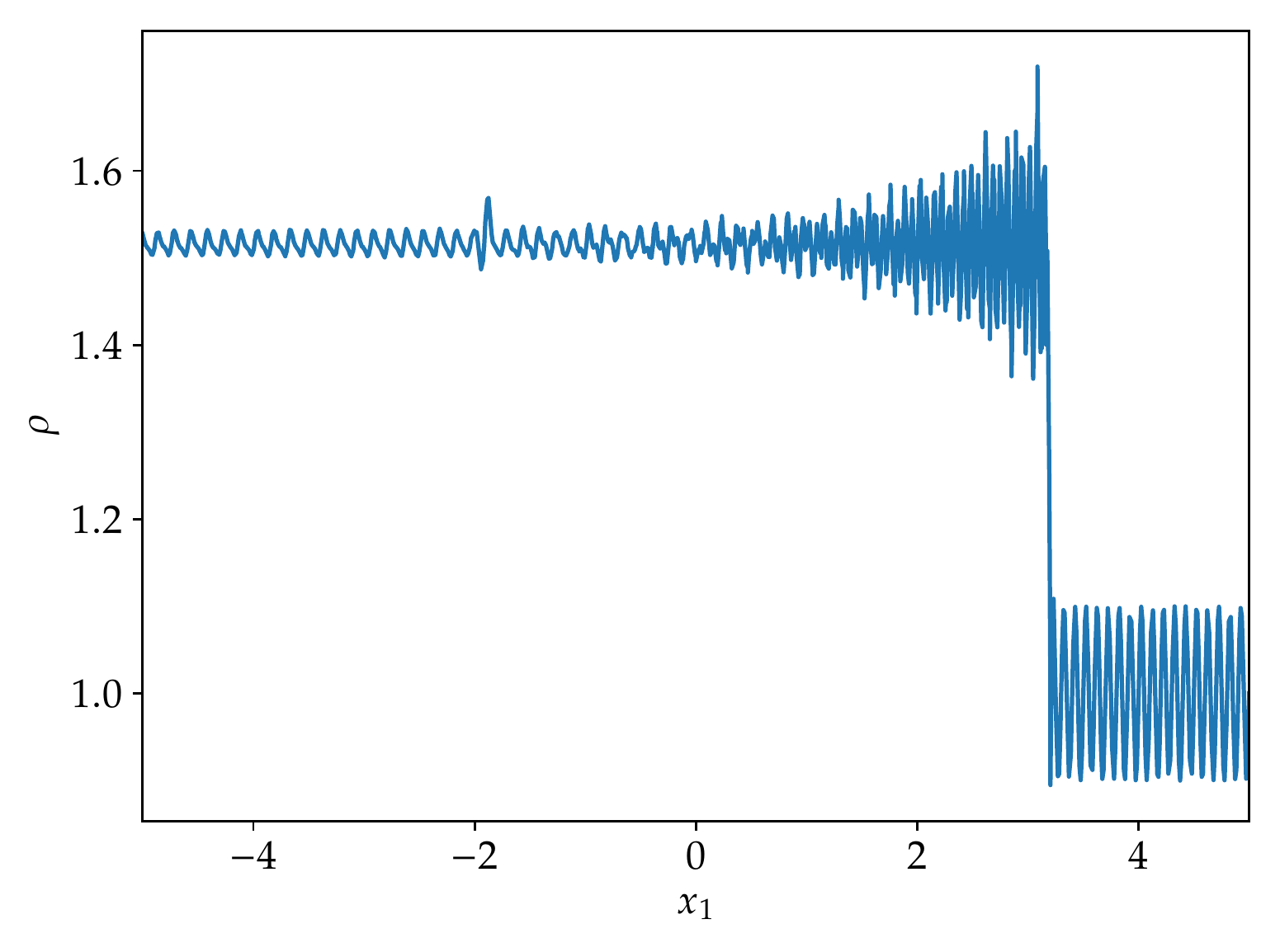}
    \caption{With relaxation.}
    \label{fig:sine_with_relaxation}
  \end{subfigure}%
  \caption{Density profile of the sine-shock interaction problem.}
  \label{fig:sine}
\end{figure}
The entropy stable semidiscretization uses polynomials of degree $p=3$ on a grid with
256 elements. The other parameters are the same as for Sod's shock
tube problem in Section~\ref{subsec:sod}.

Again, results with and without relaxation are  visually indistinguishable,
as shown for the density in Figure~\ref{fig:sine}, supporting the
conclusions of Section~\ref{subsec:sod}.
For this experiment, $\gamma$ deviates from one less than $5 \times 10^{-4}$,
as shown in Figure~\ref{fig:sine-gamma}.

\subsection{Lid-Driven Cavity Flow}\label{subsubsec:ldc}
Next, we validate the algorithm simulating a three-dimensional
lid-driven cavity flow. The domain is a cube with sides of length $l$ discretized using a
Cartesian grid composed of eight elements in each direction.
A velocity field is imposed on one of the walls,
corresponding to a rigid body rotation about the
center of the wall at an angular speed $\omega$ (see Figure \ref{fig:dc_bcs}).
We choose the rotation velocity and the size of the cavity such that
this example is characterized by
a Reynolds number $Re=l^2\omega/\nu=100$ and a Mach number $M=l \omega/c=0.05$.
All the dissipation terms used for the interface coupling \cite{parsani_entropy_stable_interfaces_2015}
and the imposition of the boundary conditions \cite{dalcin_2019_wall_bc} are
turned off, including upwind and interior-penalty SAT terms.

First, we show the performance of some relaxation Runge--Kutta schemes for the case
where entropy conservative adiabatic wall boundary conditions \cite{dalcin_2019_wall_bc}
are used on all the six faces of the cavity (see Figure
\ref{fig:dc_bc_noheat}).
Figure \ref{fig:dc_bc_noheat_eta} shows the time evolution of the discrete total entropy
$\eta=\mathbf{1}^{\top} \widehat{\Pmatvol} \, \bm{S}$.

Two highly resolved numerical solutions computed with an eighth-order accurate
scheme ($p=7$), the BSRK(8,5) and the VRK(9,6) time integration scheme
using a time adaptive algorithm with a tolerance of $10^{-8}$ are shown in
Figure \ref{fig:dc_bc_noheat_eta}. They are indistinguishable at the resolution of
the plot, and can be regarded as a reference solution. Because the solutions with and without
relaxation are very close to each other, only the results obtained with the
relaxation Runge--Kutta schemes are shown.
After a very short transient phase associated with the impulsive startup of the
rotating plate, $\eta$ decreases linearly. The reason is simple:
the imposed no-slip
wall boundary conditions on the six faces of the cavity are entropy conservative
and the only term in Equation \eqref{eq:estimate-no-slip-bc-2} which is non-zero
is $-\mathbf{DT}$. This contribution is strictly negative semi-definite
and constant because the flow at this Reynolds number is laminar and steady
and therefore, the gradient of the entropy variables in Equation \eqref{eq:DT}
does not change in time.

The results of three additional simulations with second-, third- and fourth-order
accurate solvers (again with $\Delta t=10^{-4}$) are also plotted in Figure \ref{fig:dc_bc_noheat_eta}.
For these methods, a fixed step size $\Delta t = 10^{-4}$ was used.
We find again that the entropy evolution with and without relaxation is indistinguishable.
This demonstrates that the relaxation approach gives a stability guarantee and,
unlike most numerical stabilization techniques, does not (in this case) add any
significant dissipation.
It can be clearly seen that the rate of entropy decay is different for different
entropy-conservative algorithms because of their accuracy. However, higher-order
discretizations give an entropy evolution that is closer to that of the the
reference solution.
\begin{figure}
\centering
  \begin{subfigure}[b]{.65\textwidth}
    \centering
    \includegraphics[width=\textwidth]{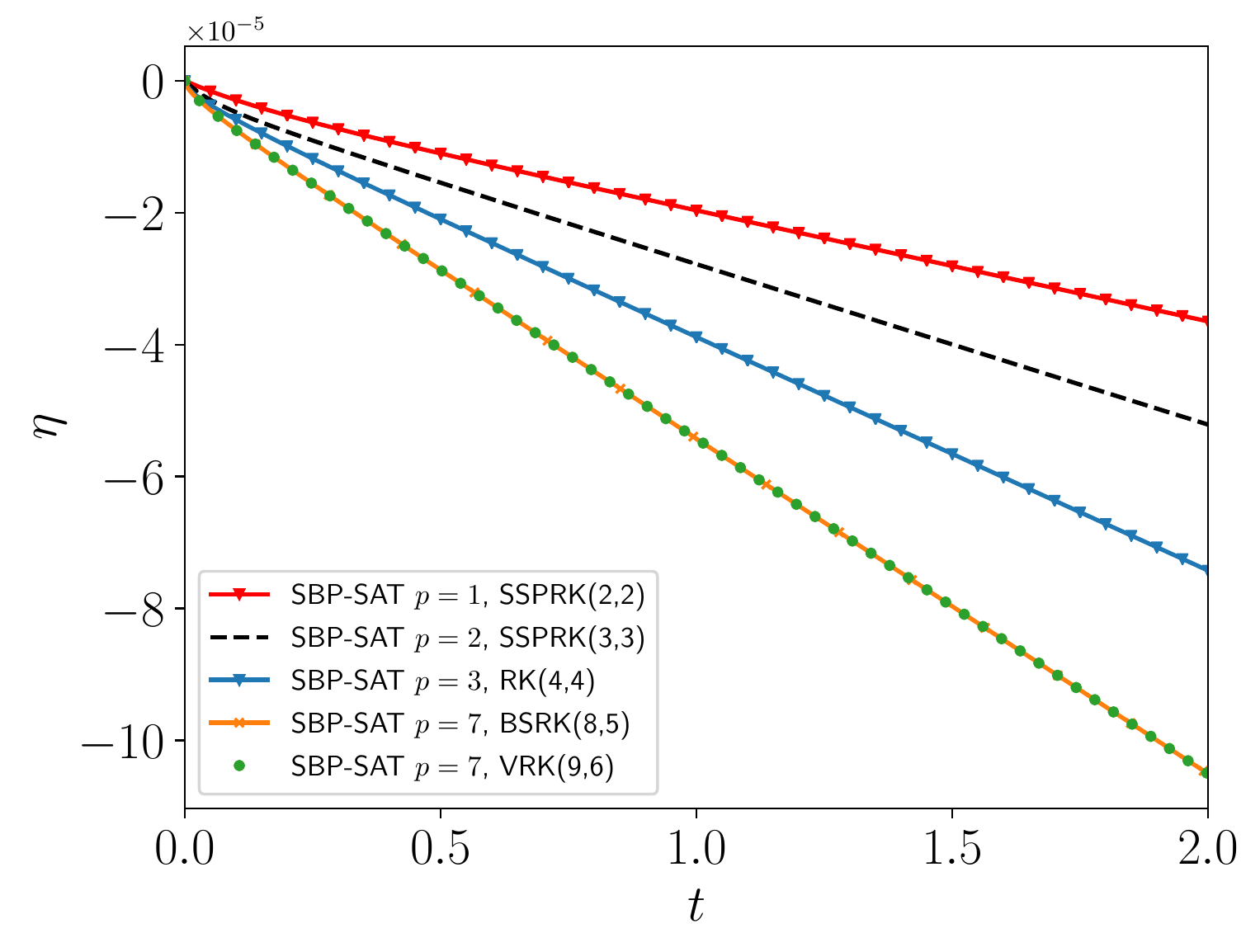}
    \caption{Time evolution of the entropy.}
    \label{fig:dc_bc_noheat_eta}
  \end{subfigure}
  \begin{subfigure}[b]{.3\textwidth}
    \centering
    \includegraphics[width=\textwidth]{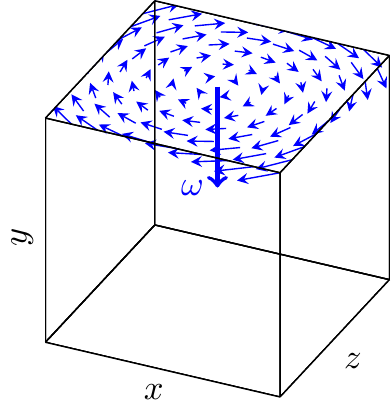}
    \caption{Sketch of the test case.}
    \label{fig:dc_bc_noheat}
  \end{subfigure}%
  \caption{Driven cavity with rigid body rotation $\omega$ and zero
  heat entropy flux.}
\end{figure}

Next we present in Figure \ref{fig:dc_bcs_eta} the results for the same set of
relaxation Runge--Kutta schemes when a non-zero heat entropy flux,
$\mathtt{g}(t)= - 10^{-4}\sin(4\pi t)$, is imposed on one of the faces adjacent
to the rotating face (see Figure \ref{fig:dc_bcs}).
Because of the added heat, the exact time evolution of the entropy is not monotonic.
This can be seen in the reference solutions provided again by using
an eighth-order accurate spatial scheme ($p=7$) with the BSRK(8,5) and the
VRK(9,6) time integration schemes using a time-adaptive algorithm with a tolerance of $10^{-8}$.
We observe that the accuracy of the entropy evolution in time depends as expected
on the order of the temporal and spatial discretizations.
Again, the entropy evolution with and without relaxation is indistinguishable,
indicating that the RRK methods do not add significant dissipation.

\begin{figure}
\centering
  \begin{subfigure}[b]{.65\textwidth}
    \centering
    \includegraphics[width=\textwidth]{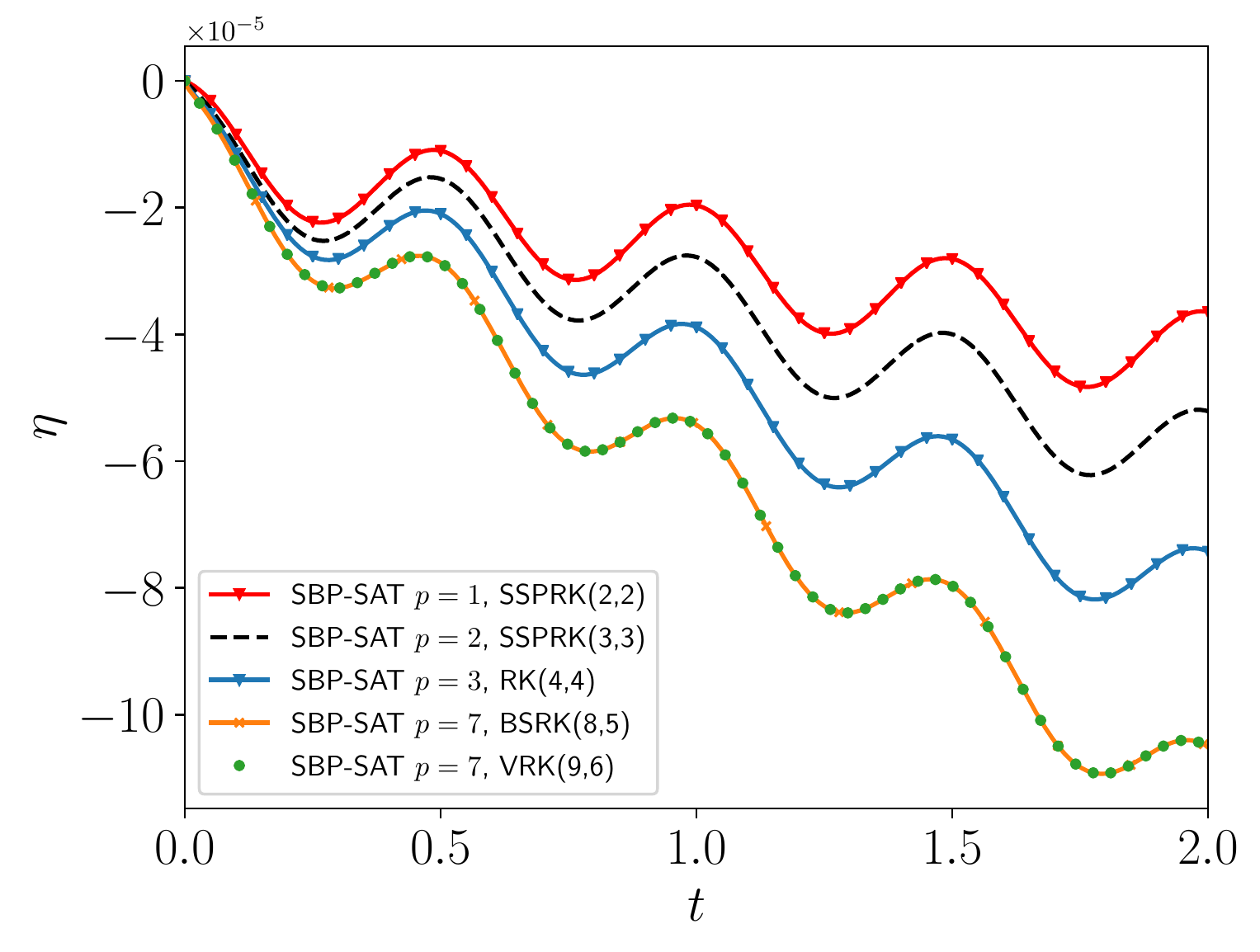}
    \caption{Time evolution of the entropy function.}
    \label{fig:dc_bcs_eta}
  \end{subfigure}
  \begin{subfigure}[b]{.3\textwidth}
    \centering
    \includegraphics[width=\textwidth]{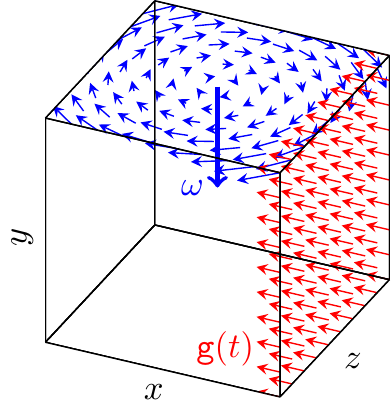}
    \caption{Sketch of the test case.}
    \label{fig:dc_bcs}
  \end{subfigure}%
  \caption{Driven cavity with rigid body rotation $\omega$ and non-zero
  heat entropy flux, $\mathtt{g}(t)$.}
\end{figure}


\subsection{Supersonic Turbulent Flow Past a Rod}\label{subsubsec:rod}
We finally provide further evidence of the robustness of the algorithm  in the
context of supersonic flow around a square cylinder with $Re_\infty=10^4$ and
$M_\infty=1.5$, which features shocks, expansion regions and three-dimensional
vortical structures \cite{parsani_entropy_stability_solid_wall_2015}.
The three-dimensional mesh used in the study consists of 87,872 hexahedral elements.
The boundary conditions imposed are adiabatic solid wall on the square cylinder surfaces
\cite{dalcin_2019_wall_bc}, periodic boundary
conditions in the $x_3$ direction, and far field at the remaining boundaries.
The problem is solved using a fourth-order accurate ($p=3$) spatial discretization and
RK(4,4) with relaxation.

\begin{figure}
  \centering
  \begin{tabular}{c c}
    \includegraphics[width = 0.45\textwidth]{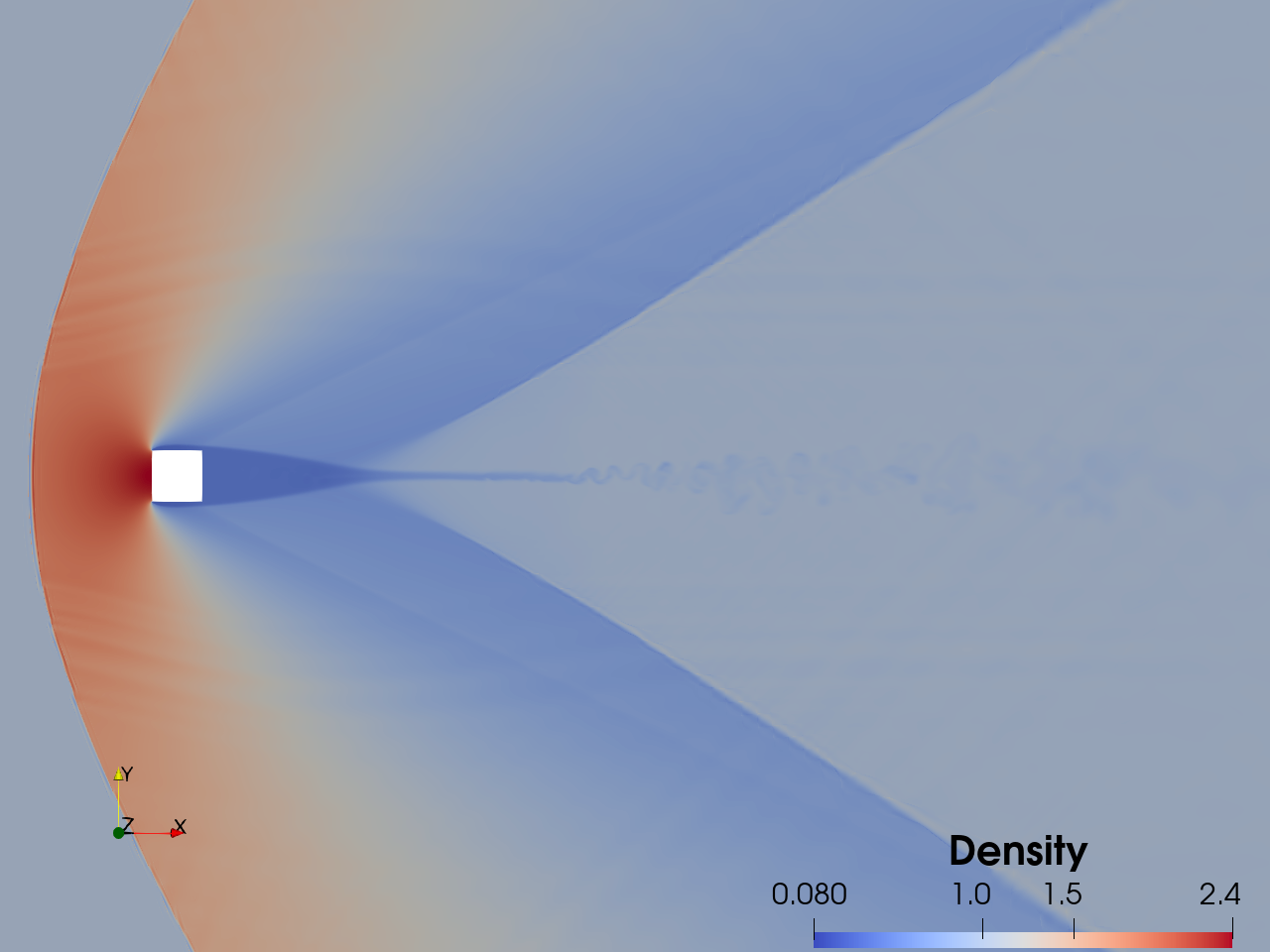} &
    \includegraphics[width = 0.45\textwidth]{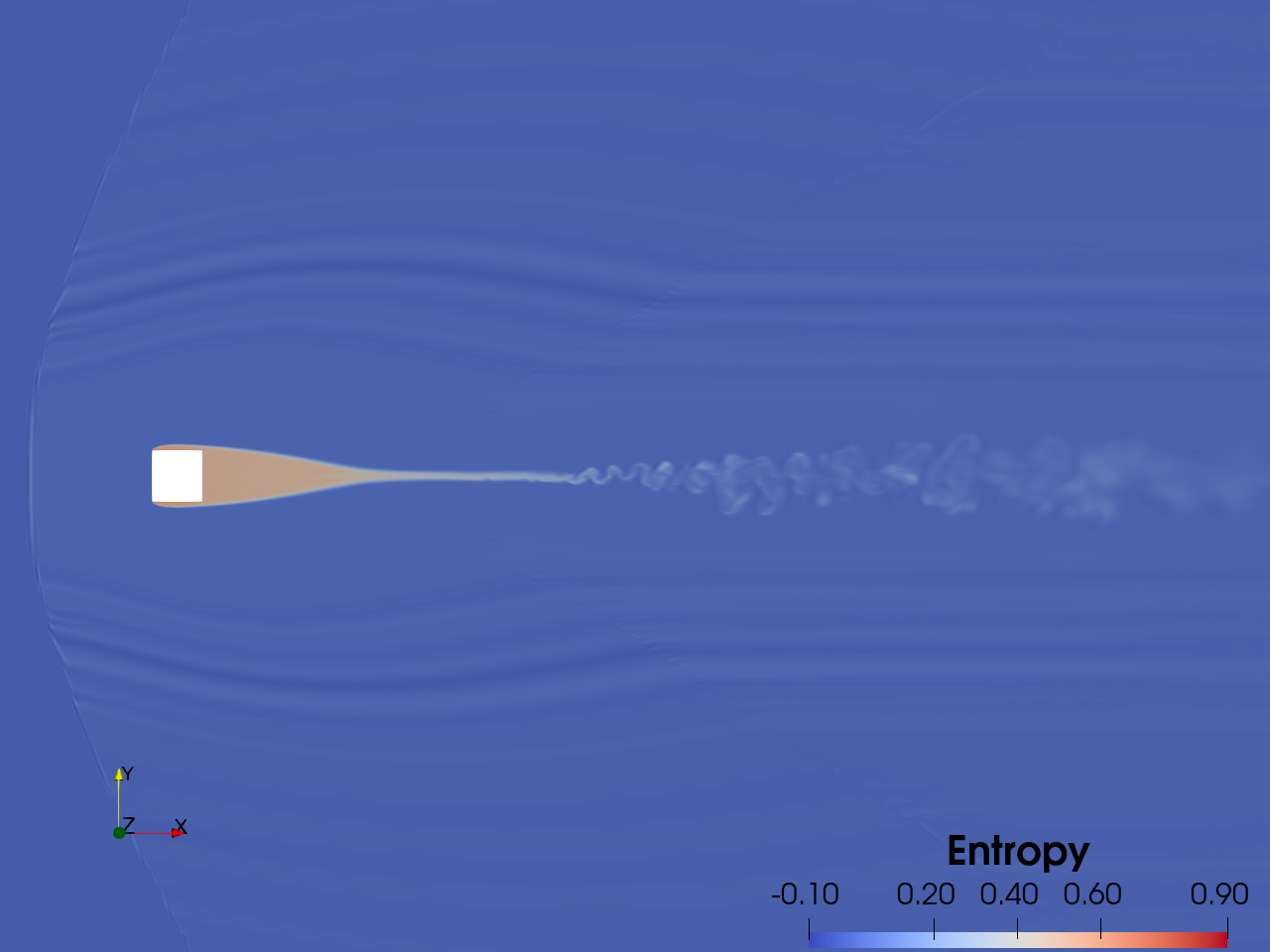}\\
    \includegraphics[width = 0.45\textwidth]{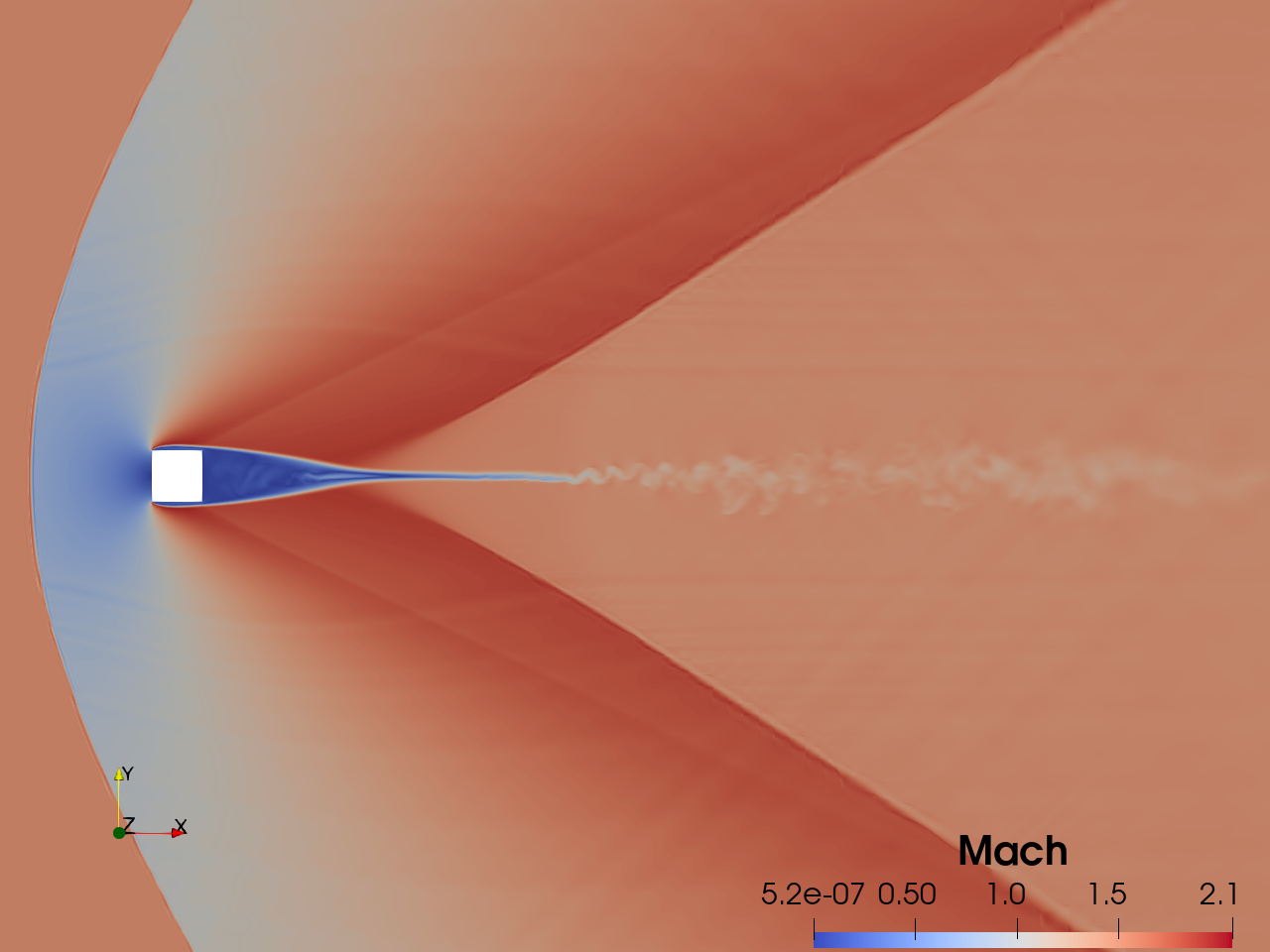} &
    \includegraphics[width = 0.45\textwidth]{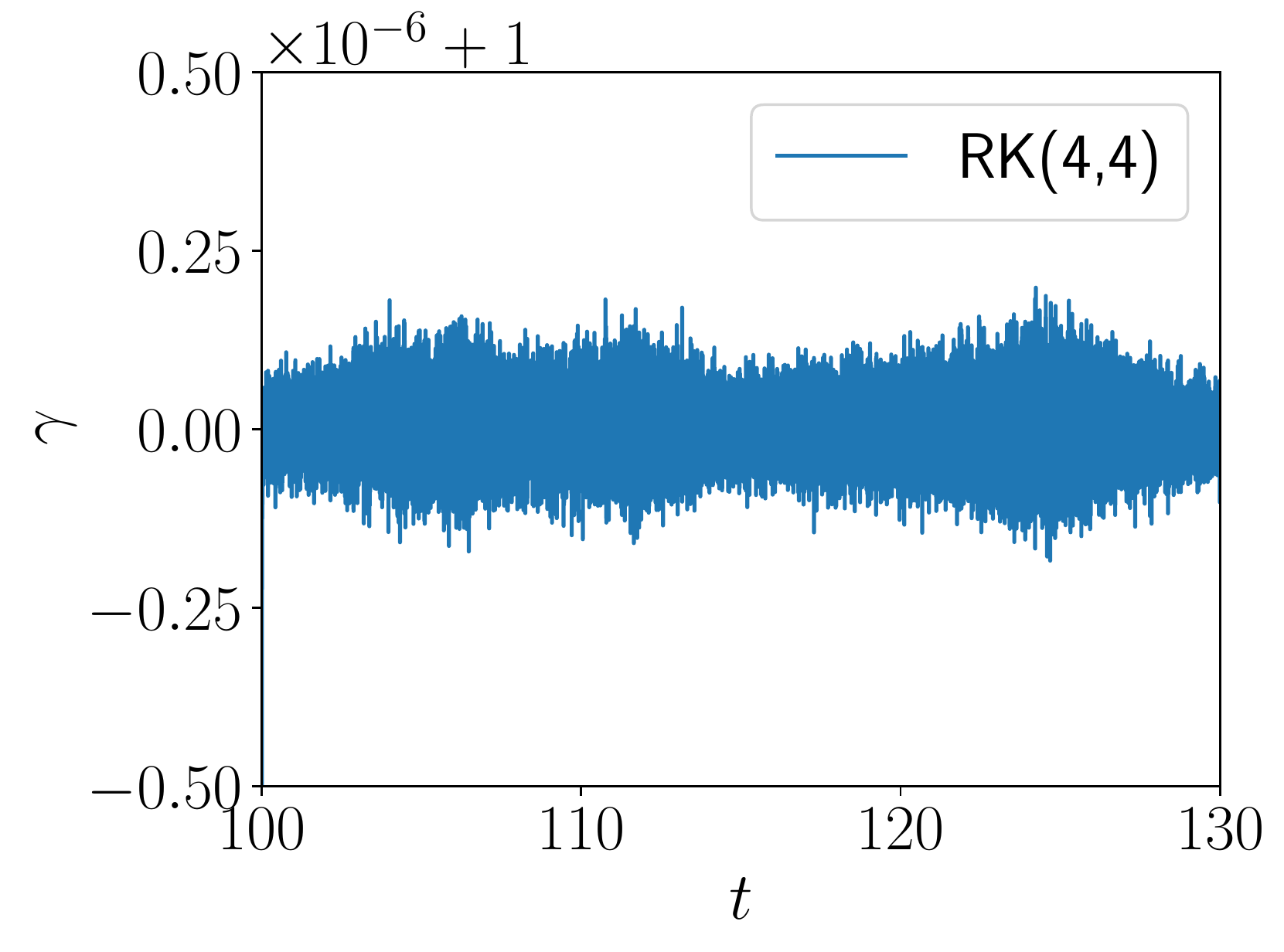}\\
  \end{tabular}
  \caption{Density, thermodynamic entropy, local Mach number contours and snapshot of the time evolution
  of the relaxation parameter, $\gamma$, for the supersonic flow around a square cylinder with $Re_\infty=10^4$ and $M_\infty=1.5$ at $t=130$.}
  \label{fig:ssresults}
\end{figure}
Figure \ref{fig:ssresults} shows the results for the
supersonic square cylinder at $t=130$ and the time evolution of the
relaxation factor, $\gamma$. At $t=130$, the flow is fully unsteady
and the shock in front of the cylinder has reached its final position. The flow is characterized by the shock in
front of the square cylinder and those in the near wake region. There is also an unsteady wake populated by three-dimensional vortices
shedding from the body.
The time evolution of the relaxation factor shows that the value of $\gamma$
oscillates around one with a maximum deviation from it of $2.5 \times 10^{-7}$.

We finally remark that the small oscillations near the shock
region are caused by discontinuities in the solution and are expected for this scheme.
In fact, we are not using any shock capturing method or reducing the order of
scheme at the discontinuity. Nevertheless, the simulation remains stable at all
time, and the oscillations are always confined to small regions near the
discontinuities. This is a feat unattainable with several alternative approaches
based on linear analysis which for this test problem lead to
numerical instabilities and an almost immediate crash of the solver
\cite{parsani_entropy_stability_solid_wall_2015}.

\section{Conclusions}
In this paper we have proposed, analyzed, and demonstrated a general approach which
allows any Runge--Kutta method to preserve the correct time evolution of an
arbitrary functional, without sacrificing linear covariance, accuracy, or
stability properties of the original method. In the case of convex functionals,
there are additional insights such as the possibility to add entropy dissipation
by the time integration scheme. This and procedures for adaptive time step
controller will be studied deeper in the future.  We are also studying the
impact of relaxation on the stable time step size.

The new approach, combined with an appropriate
entropy-conservative/entropy-stable semi-discretization on unstructured grids,
yields the \emph{first discretization for computational fluid dynamics} that is:
\begin{itemize}
    \item Primary conservative
    \item Entropy-conservative/entropy-stable in the fully-discrete sense
      with $\Delta t = {\mathcal O}(\Delta x)$
    \item Explicit, except for the solution of a scalar algebraic equation at each time step
    \item Arbitrarily high-order accurate in space and time
\end{itemize}
Furthermore, the added computational cost of this modification is insignificant
in the context of typical computational fluid dynamics calculations. It is
anticipated that this type of entropy stable formulation will begin to
bear fruit for industrial simulations in the near future \cite{cfd_2030_nasa}.
Finally, relaxation schemes provide an entropy guarantee without
degrading solution accuracy or adding unnecessary dissipation.

Further desirable properties of fully discrete numerical methods for the compressible Euler
and Navier--Stokes equations not studied in this article concern additional elements of robustness, e.g.
preserving the positivity of two thermodynamic variables (e.g., density and pressure). To use the framework of
\cite{zhang2011maximum}, the interplay of limiters and relaxation schemes has
to be studied.

Moreover, having a local entropy (in-)equality instead of the global one established
in this article might be advantageous. However, this seems to be currently out of
reach using the relaxation schemes proposed here. While fully-discrete local entropy
inequalities can be achieved by the addition of sufficient artificial viscosity,
the advantage of relaxation schemes is that they do not impose excessive dissipation;
if the baseline scheme is dissipative, they can even remove some of this dissipation.

\section*{Acknowledgments}
The research reported in this paper was funded by King Abdullah University of
Science and Technology. We are thankful for the computing resources of the
Supercomputing Laboratory and the Extreme Computing Research Center at
King Abdullah University of Science and Technology.